\documentclass{amsart}
\usepackage[french,english]{babel}
\usepackage[utf8]{inputenc}
\usepackage{amsmath, amssymb}
\input xy
\xyoption{all}

\newtheorem{theorem}{Theorem}[section]
\newtheorem{lemma}[theorem]{Lemma}
\newtheorem{corollary}[theorem]{Corollary}
\newtheorem{fact}[theorem]{Fact}
\newtheorem{proposition}[theorem]{Proposition}

\theoremstyle{definition}
\newtheorem{example}[theorem]{Example}
\newtheorem{remark}[theorem]{Remark}
\newtheorem{definition}[theorem]{Definition}

\def\End{\operatorname{End}}
\def\alg{\operatorname{alg}}
\def\id{\operatorname{id}}
\def\ord{\operatorname{ord}}

\begin{document}

\title{$F$-sets and finite automata}

\author{Jason Bell}
\address{Jason Bell\\
University of Waterloo\\
Department of Pure Mathematics\\
200 University Avenue West\\
Waterloo, Ontario \  N2L 3G1\\
Canada}
\email{jpbell@uwaterloo.ca}

\author{Rahim Moosa}
\address{Rahim Moosa\\
University of Waterloo\\
Department of Pure Mathematics\\
200 University Avenue West\\
Waterloo, Ontario \  N2L 3G1\\
Canada}
\email{rmoosa@uwaterloo.ca}



\thanks{{\em Acknowledgements}: Both authors were partially supported by their NSERC Discovery Grants. The first author was also supported by the European Research Council (ERC) under the European Union's Horizon 2020 research and innovation programme under the Grant Agreement No 648132.}

\date{\today}

\selectlanguage{english} 

\begin{abstract}
It is observed that Derksen's Skolem-Mahler-Lech theorem is a special case of the isotrivial positive characteristic Mordell-Lang theorem due to the second author and Scanlon.
This motivates an extension of the classical notion of a $k$-automatic subset of the natural numbers to that of an {\em $F$-automatic subset} of a finitely generated abelian group $\Gamma$ equipped with an endomorphism $F$.
Applied to the Mordell-Lang context, where $F$ is the Frobenius action on a commutative algebraic group $G$ over a finite field, and $\Gamma$ is a finitely generated $F$-invariant subgroup of $G$,
it is shown that the ``$F$-subsets" of $\Gamma$ introduced by the second author and Scanlon are $F$-automatic.
It follows that when $G$ is semiabelian and $X\subseteq G$ is a closed subvariety then $X\cap \Gamma$ is $F$-automatic.
Derksen's notion of a $k$-normal subset of the natural numbers is also here extended to the above abstract setting, and it is shown that $F$-subsets are $F$-normal.
In particular, the $X\cap \Gamma$ appearing in the Mordell-Lang problem are $F$-normal.
This generalises Derksen's Skolem-Mahler-Lech theorem to the Mordell-Lang context.
\end{abstract}
\maketitle

\selectlanguage{french} 
\begin{abstract}
\noindent
On observe que le th\'eor\`eme de Skolem-Mahler-Lech de Derksen est un cas particulier du th\'eor\`eme de Mordell-Lang isotrivial en caract\'eristique positive d\^u au second auteur et Scanlon. Cela motive une extension de la notion classique d'un sous-ensemble $k$-automatique des nombres naturels \`a celle d'un ensemble $F$-automatique d'un groupe ab\'elien de type fini $\Gamma$ \'equip\'e d'un endomorphisme $F$. Dans le contexte de Mordell-Lang, o\`u $F$ est l'action de Frobenius sur un groupe alg\'ebrique commutatif $G$ sur un corps fini, et $\Gamma$ est un sous-groupe $F$-invariant de $G$, il est montr\'e que les ``$F$-sous-ensembles" de $\Gamma$ introduits par le second auteur et Scanlon sont $F$-automatiques. Il en d\'ecoule que lorsque $G$ est semi-ab\'elien et $X\subseteq G$ est une sous-vari\'et\'e ferm\'ee, $X\cap\Gamma$ est $F$-automatique. La notion d'un sous-ensemble $k$-normal des nombres naturels au sens de Derksen est \'egalement g\'en\'eralis\'ee au contexte abstrait cit\'e ci-dessus, et il est d\'emontr\'e que les $F$-sous-ensembles sont $F$-normaux. En particulier, les ensembles $X\cap\Gamma$ qui apparaissent dans le probl\`eme de Mordell-Lang sont $F$-normaux. Cela g\'en\'eralise le th\'eor\`eme de Skolem-Mahler-Lech de Derksen au contexte de Mordell-Lang.
\end{abstract}

\maketitle
\vfill\pagebreak
\selectlanguage{english}

\setcounter{tocdepth}{1}
\tableofcontents

\section{Introduction}

\noindent
This paper is concerned with the structure of sets arising in the diophantine geometric context of the Mordell-Lang theorem; sets of the form $X\cap\Gamma$ where $X$ is a closed subvariety of a commutative algebraic group $G$ and $\Gamma$ is a finitely generated subgroup of $G$.
When $G$ is a semiabelian variety in characteristic zero, the structure of such sets is described by Faltings' theorem~\cite{faltings}.
In positive characteristic, Hrushovski~\cite{udi} proved a function field version that treats as exceptional the case when $G$ is defined over a finite field.
That exceptional case, called the {\em isotrivial} case, is dealt with by the second author and Scanlon in~\cite{fsets} by giving a combinatorial description of $X\cap \Gamma$ as {\em $F$-sets}, essentially finite unions of translates of subgroups by orbits under the Frobenius endomorphism -- see Definition~\ref{defnfsets} below for details.
Now, the Skolem-Mahler-Lech theorem on the vanishing set of a linear recurrence can also be recast in these terms.
In positive characteristic it amounts to the case when $G$ is a power of the multiplicative torus, $\Gamma$ is cyclic, and $X$ is linear.
The fundamental theorem here is that of Derksen~\cite{derksen}, which can be viewed as giving a description of $X\cap\Gamma$ in terms of finite state automata.
On the face of it, these two types of descriptions, $F$-sets versus automatic sets, have little to do with each other.
However, in this paper, we effect a unification, showing how to interpret $F$-sets in terms of finite state machines in a very general setting.

In fact, Derksen's Skolem-Mahler-Lech theorem of~\cite{derksen} can be deduced from the isotrivial Mordell-Lang theorem of~\cite{fsets}.
This seems to have escaped the notice of Derksen, though he does mention~\cite{fsets} when discussing related problems.
The initial impetus for writing this paper was to make that argument explicit, translating between the language of these two papers.
This is done in Section~\ref{fsets-to-sml}.

We are lead to ask how $F$-sets relate to automaticity, not only in the Skolem-Mahler-Lech context, but rather in the more general setting of the isotrivial Mordell-Lang.
In Section~\ref{section-fset-fauto} we take a first step by showing that for any finitely generated abelian group $\Gamma$, if $F$ is a positive integer acting on $\Gamma$ by multiplication, then every $F$-subset of $\Gamma$ is $F$-automatic.
This is Theorem~\ref{fset-fauto} below.

But the isotrivial Mordell-Lang context includes situations when the endomorphism is not multiplication by an integer, for example when it is the Frobenius action on the rational points of an abelian variety over a finite field.
We are lead to consider the general setting of a finitely generated abelian group $\Gamma$ equipped with an injective endomorphism $F$.
In order to make sense of what an $F$-automatic subset of $\Gamma$ should be, we consider ``base-$F$" expansions of the form
$$[x_0\cdots x_m]_F:=x_0+Fx_1+\cdots F^mx_m$$
where the $x_i$ are ``digits" from a fixed finite subset $\Sigma\subseteq\Gamma$.
We call $\Sigma$ an {\em $F$-spanning set} if, among other more or less natural conditions, every element of $\Gamma$ can be written (possibly non-uniquely) in this way -- see Definition~\ref{spanning} below for details.
Section~\ref{section-spanning} is dedicated to studying properties of $F$-spanning sets, and deriving sufficient criteria for their existence.
Proposition~\ref{height} says that if $\Gamma$ admits an appropriate height function then at least an $F^r$-spanning set exists for some $r>0$.
In particular, this is the case when $F$ is the Frobenius endomorphism on a commutative algebraic group $G$ over a finite field and $\Gamma$ is an $F$-invariant finitely generated subgroup of~$G$.

Once we have an $F$-spanning set, $\Sigma$, we can define a subset of $\Gamma$ to be {\em $F$-automatic} if it is of the form $\{[w]_F:w\in\mathcal L\}$ for some regular language $\mathcal L\subseteq \Sigma^*$.
Here $\Sigma^*$ denotes the set of finite strings on $\Sigma$, and a regular language is a subset that is the set of words accepted by some finite automaton.
$F$-automaticity is explored in Section~\ref{section-genfset}, and it is shown to be independent of the $F$-spanning set $\Sigma$.
Moreover, one can work under the weaker hypothesis that an $F^r$-spanning set exists for some $r>0$.
The main theorem here is that under some mild conditions (satisfied, for example, by the Frobenius action on commutative algebraic groups over a finite field), $F$-sets are $F$-automatic.
This is Theorem~\ref{generalfsetfauto} below.
Putting this together with~\cite{fsets} we get that, when $G$ is semiabelian, if $X\subseteq G$ is a closed subvariety then $X\cap \Gamma$ is $F$-automatic.

Toward a more precise automata-theoretic explanation of $F$-sets we extend Derksen's notion of {\em $p$-normality} to this abstract setting in Section~\ref{section-fnormal}.
This is done using the notion of a sparse language, see Proposition~\ref{sparse} for various formulations of sparseness.
Theorem~\ref{fset=fnorm} says that the $F$-sets are of the form $\gamma+T+H$ where $\gamma\in \Gamma$, $H$ is an $F$-invariant subgroup, and $T=\{[w]_{F^r}:w\in\mathcal L\}$  for a sparse regular language $\mathcal L$ on an $F^r$-spanning set.

It is important to note that Derksen~\cite{derksen} is able to effectively determine the vanishing sets of linear recurrences, while no such effectivity can be deduced directly from~\cite{fsets}.
Given the finite state machine explanation of $F$-sets developed here, one can hope for an effective version of the isotrivial positive characteristic Mordell-Lang theorem.
This will be the subject of a subsequent paper.

\bigskip
\section{Isotrivial Mordell-Lang in Characteristic $p$}
\label{iml}

\noindent
The Mordell-Lang theorem of Faltings in characteristic zero, as it appears in~\cite{faltings}, says that the intersection of a closed subvariety $X$ of a semiabelian variety $G$ with a finitely generated subgroup $\Gamma$ is a finite union of cosets of subgroups of $\Gamma$.
Since the Zariski closure of a subgroup of $\Gamma$ is an algebraic subgroup of $G$, one can see this theorem as saying that if $X\cap\Gamma$ is infinite it is because of a positive dimensional algebraic subgroup of $G$ that is contained in a translate of $X$.
It is not difficult to find counterexamples in characteristic $p$, where the Frobenius endomorphism on semiabelian varieties over finite fields (that is, {\em isotrivial} semiabelian varieties) can be another source for the infinitude of $X\cap \Gamma$.
Hrushovski salvaged the situation in~\cite{udi} by proving a function field version in positive characteristic that treats the isotrivial case as exceptional.

The second author and Scanlon dealt with this exceptional case in~\cite{fsets}, describing $X\cap \Gamma$ as an ``$F$-set" -- something we now define in an abstract setting.

\begin{definition}[$F$-sets]
\label{defnfsets}
Suppose $\Gamma$ is a finitely generated abelian group equipped with an endomorphism $F$.
An {\em $F$-subset} of $\Gamma$ is defined to be a finite union of finite set sums of
\begin{itemize}
\item
singletons in $\Gamma$,
\item
$F$-invariant subgroups of $\Gamma$, and
\item
{\em $F$-cycles} in $\Gamma$; namely, sets of the form
$$C(\gamma;\delta):=\{\gamma+F^\delta\gamma+F^{2\delta}\gamma+\cdots+F^{\ell\delta}\gamma:\ell<\omega\}$$
where $\gamma\in \Gamma$ and $\delta$ is a positive integer.
\end{itemize}
Here, and throughout this paper, $\omega$ denotes the least infinite ordinal; the set of all natural numbers including zero.
\end{definition}

While much of~\cite{fsets} is dedicated to the combinatorics and first-order properties of $F$-sets, the main theorem of that paper is the following:

\begin{theorem}[Isotrivial Mordell-Lang~\cite{fsets}]
\label{ml}
Suppose $G$ is a semiabelian variety over a finite field $\mathbb F_q$ and $F:G\to G$ is the endomorphism induced by the $q$-power Frobenius map.
Suppose $X$ is a  closed subvariety of $G$ and $\Gamma$ is a finitely generated subgroup of $G(K)$ where $K$ is a finitely generated regular extension of $\mathbb F_q$, and that $\Gamma$ is preserved by~$F$.
Then $X\cap \Gamma$ is an $F$-subset of $\Gamma$.
\end{theorem}

Some later refinements of this theorem include: \cite{ghioca} where the assumption that $\Gamma$ is preserved under $F$ is dropped and where a version for the additive group is formulated, \cite{ghiocamoosa} where it is extended to the case when $\Gamma$ is of finite rank (i.e., in the divisible hull of a finitely generated group), and \cite{fsets2} where it is combined with Hrushovski's theorem to give a general (not necessarily isotrivial) description of the structure of $X\cap \Gamma$ in characteristic $p$.

It may be worth remarking that the assumption in the theorem that $\Gamma\leq G(K)$ where $K/\mathbb F_q$ is regular, is not too harmful.
Indeed, we can always replace $\mathbb F_q$ by $K\cap\mathbb F_q^{\alg}=\mathbb F_{q^r}$ for some $r>0$ to ensure regularity.
So by replacing $F$ with $F^r$ we can still apply the theorem.
Now, replacing $F$ by $F^r$ does make a difference; for example it changes which subgroups are allowed in the definition of $F$-sets as not every $F^r$-invariant subgroup need be $F$-invariant.
However, at least at the level of $F$-cycles it makes no difference.
What we mean is that not only is every $F^r$-cycle an $F$-cycle by definition, but we even have a kind of converse under some mild conditions:

\begin{fact}
\label{rmult}
Suppose $\Gamma$ is a finitely generated abelian group and $F$ is an endomorphism of $\Gamma$ with the property that, for any $\delta>0$, $F^{\delta}-1$ is not a zero divisor in the subring of $\End(\Gamma)$ generated by $F$.
Then, for any $r>0$, any $F$-cycle of the form $C(\gamma;\delta)$ can be written as a finite union of translates of $F$-cycles of the form $C(\gamma';r\delta)$.

In particular, every $F$-cycle is a finite union of translates of $F^r$-cycles.
\end{fact}

\begin{proof}
This is a slightly technical but not difficult argument that is done in~\cite{fsets}, but we give some details.
The idea is to first pass to an extension $\Gamma'\geq\Gamma$ where there is $\beta\in \Gamma'$ such that $F^\delta\beta-\beta=\gamma$.
This uses the assumption on $F$ that we have made, and is done in the proof of Lemma~2.7 of~\cite{fsets}.
It is pointed out there that then
$$C(\gamma;\delta)=-\beta+\{F^{\ell\delta}(F^{\delta}\beta):\ell<\omega\}.$$
That is, in $\Gamma'$, we have written $C(\gamma;\delta)$ as a translate of an $F^\delta$-orbit.
Next, one splits the orbit appearing in this expression into a finite union of orbits under $F^{r\delta}$ depending on $\ell$mod$r$.
That is,
$$C(\gamma;\delta)=\bigcup_{m=0}^{r-1}-\beta+\{F^{\ell r\delta}(F^{(m+1)\delta}\beta):\ell<\omega\}.$$
Finally one rewrites the new translated $F^{r\delta}$-orbits in $\Gamma'$ as $F$-cycles back in $\Gamma$ but now with respect to $r\delta$.
See the proof of Lemma~2.9 of~\cite{fsets} for details.
What one gets is that the translated orbit $-\beta+\{F^{\ell r\delta}(F^{(m+1)\delta}\beta):\ell<\omega\}$ is equal to the union of the singleton $\{-\beta+F^{(m+1)\delta}\beta\}$ together with the translated $F$-cycle
$$-\beta+F^{(m+1)\delta}\beta \ +\ C(F^{(r+m+1)\delta}\beta-F^{(m+1)\delta}\beta;r\delta).$$
One then observes that the parameters here are in fact all in $\Gamma$.
\end{proof}

\bigskip
\section{Deriving Skolem-Mahler-Lech}
\label{fsets-to-sml}
\noindent
In this section we point out that when specialised to the case of the integers equipped with the multiplication by $p$, the $F$-sets coincide precisely with Derksen's $p$-normal sets from~\cite{derksen}.
We then show how to deduce the characteristic $p$ Skolem-Mahler-Lech from Theorem~\ref{ml} above.

So consider now the very special case when $\Gamma=(\mathbb Z,0,+)$ and $F$ is multiplication by a positive integer $p$.
Note that every subgroup is preserved by $F$, so the $\mathbb Z[F]$-submodules that appear in the definition of $F$-sets are replaced now simply by subgroups.
One consequence of this is that, by Fact~\ref{rmult}, the $F$-subsets of $\mathbb Z$ coincide with the $F^r$-subsets, for any positive integer $r$.

This special case actually appears, though multiplicatively rather than additively, in the isotrivial Mordell-Lang context: take $G=\mathbb G_m^\mu$ a multiplicative torus in characteristic $p$ (now a prime number), $F$ the endomorphism of raising to the $p$th power, and $\Gamma$ the cyclic subgroup generated by some $\alpha=(\alpha_1,\dots,\alpha_\mu)\in\mathbb G_m^\mu(K)$ where $K$ is a field of characteristic~$p$ and not all the $\alpha_i$ are roots of unity.
As explained in the previous section, to meet the regularity hypothesis of Theorem~\ref{ml} we may have to replace $F$ by some $F^r$, but as we have just explained, this makes no difference in this case.

So what are the $F$-subsets of $\mathbb Z$?
We show, using some results from~\cite{fsets}, that they are nothing other than Derksen's $p$-normal sets, which we now recall.\footnote{In~\cite{derksen} $p$-normal sets are defined as subsets of $\mathbb N$, but the extension to $\mathbb Z$ that we presented here is the natural generalisation. We have also adjusted the notation slightly.}

\begin{definition}[$p$-normal sets]
\label{pnormal}
An {\em elementary $p$-nested} subset of the integers is a set of the form
$\tilde S_{p^\delta}(c_0;c_1,\dots,c_d)$
where $\delta$ is a positive integer, $c_0,\dots,c_d\in\mathbb Z$ with the property that $p^\delta-1$ divides $c_0+\cdots+c_d$, and
$$\tilde S_{p^\delta}(c_0;c_1,\dots,c_d):=\{\frac{c_0}{p^\delta-1}+p^{\ell_1\delta}\frac{c_1}{p^\delta-1}+\dots+p^{\ell_d\delta}\frac{c_d}{p^\delta-1}:\ell_1,\dots,\ell_d<\omega\}.$$
A {\em $p$-normal} subset of $\mathbb Z$ is then a set which up to a finite symmetric difference is a finite union of elementary $p$-nested subsets and cosets of subgroups of $\mathbb Z$.
\end{definition}

Note that as $p^\delta-1$ divides $c_0+\cdots+c_d$, we do have $\tilde S_{p^\delta}(c_0;c_1,\dots,c_d)\subseteq\mathbb Z$.

\begin{proposition}
\label{fset=pnorm}
Suppose $F:\mathbb Z\to\mathbb Z$ is multiplication by a positive integer~$p$.
Up to finite symmetric differences, the $F$-subsets of $\mathbb Z$ are exactly the $p$-normal subsets.
\end{proposition}

\begin{proof}
We first show that $F$-sets are $p$-normal.
Taking finite unions it suffices to show that an $F$-set of the form
$$U=k_0+C(k_1;\delta_1)+\cdots+C(k_d;\delta_d)+H$$
is $p$-normal, where $k_0,\dots,k_d\in\mathbb Z$, $\delta_1,\dots,\delta_d$ are positive integers, and $H$ is a subgroup of $\mathbb Z$.
In the case when $H$ is not trivial, since $\mathbb Z/H$ is finite, we get that $U$ is in fact a finite union of cosets of $H$ and hence is $p$-normal by definition.
We may therefore assume that $H=(0)$.
Moreover, by Fact~\ref{rmult}, after replacing each $\delta_i$ with $\delta:=\delta_1\cdots\delta_d$, and taking translates and finite unions, we may assume that all the $\delta_i$'s are the same.
So we have reduced to the case when
$$U=k_0+C(k_1;\delta)+\cdots+C(k_d;\delta).$$
Now,
\begin{eqnarray*}
C(k_i;\delta)
&=&
\{k_i+F^\delta k_i+\cdots+F^{\ell\delta}k_i:\ell<\omega\}\\
&=&
\{k_i+p^\delta k_i+\cdots+p^{\ell\delta}k_i:\ell<\omega\}\\
&=&
\left\{\frac{(p^{\delta} -1)(k_i+p^{\delta} k_i+\cdots+p^{\ell \delta}k_i)}{p^{\delta}-1}:\ell<\omega\right\}\\
&=&
\left\{\frac{-k_i+p^{(\ell+1)\delta}k_i}{p^{\delta}-1}:\ell<\omega\right\}\\
&=&
\left\{\frac{-k_i}{p^{\delta}-1}+p^{\ell \delta}\frac{p^{\delta} k_i}{p^{\delta}-1}:\ell<\omega\right\}\\
&=&
\tilde S_{p^{\delta}}(-k_i;p^{\delta} k_i).
\end{eqnarray*}
Hence,
\begin{eqnarray*}
U
&=&
k_0+C(k_1;\delta)+\cdots+C(k_d;\delta)\\
&=&
k_0+ \tilde S_{p^{\delta}}(-k_1;p^{\delta} k_1)+\cdots++ \tilde S_{p^{\delta}}(-k_d;p^{\delta} k_d)\\
&=&
\tilde S_{p^{\delta}}(k_0(p^{\delta}-1)-k_1-\cdots-k_d; p^{\delta} k_1,\dots,p^{\delta} k_d).
\end{eqnarray*}
Note that the sum of the parameters in this expression is divisible by $p^{\delta}-1$ as required in the definition of an elementary $p$-nested set.
We have shown that $U$ is an elementary $p$-nested set.
Hence $F$-sets are $p$-normal.

For the converse, assume that $U$ is a $p$-normal subset of $\mathbb Z$.
After replacing $U$ with something of finite symmetric difference with it, and taking finite unions, we may assume that $U$ is either a coset of a subgroup or an elementary $p$-nested set.
The former are $F$-subsets by definition, and so it suffices to consider elementary $p$-nested sets $U=\tilde S_{p^\delta}(c_0;c_1,\dots,c_d)$.
Working in the extension $\Gamma':=\frac{1}{p^\delta-1}\mathbb Z$, note that $\tilde S_{p^\delta}(c_0;c_1,\dots,c_d)$ is a translate (by $\frac{c_0}{p^\delta-1}$) of a sum of orbits (of the $\frac{c_i}{p^\delta-1}$) under the action of multiplication by $p^\delta$.
That is, $U$ is a translate of a sum of $F^\delta$-orbits in $\Gamma'$.
Now, by Lemma~2.9 of~\cite{fsets}, and indeed by the arguments we saw in the proof of Fact~\ref{rmult} above, every translate of a sum of $F^\delta$-orbits in $\Gamma'$ that happens to lie in an $F$-invariant subgroup (in this case $\mathbb Z$) is a finite union of translates of sums of $F$-cycles of that subgroup.
Hence $U$ is a finite union of translates of sums of $F$-cycles in $\mathbb Z$, and hence is an $F$-set.
\end{proof}

We can now deduce the positive characteristic Skolem-Mahler-Lech theorem from the isotrivial Mordell-Lang.
This theorem was proved by Derksen in~\cite{derksen} a few years after Theorem~\ref{ml} appeared in~\cite{fsets}, but by entirely different methods.
Moreover, Derksen obtains an effectivity result that does not follow from the methods in~\cite{fsets}.

The original Skolem-Mahler-Lech in characteristic zero, and indeed Derksen's version in positive characteristic, has to do with $\mathbb N$-indexed sequences satisfying a linear recurrence.
For our purposes it is more convenient to talk about groups and so we will formulate it for $\mathbb Z$-indexed sequences.
That is, we have a field $K$ and a sequence $a=(a_n:n\in\mathbb Z)$ from $K$ such that for some positive integer $d$, and constants $c_0,\dots,c_{d-1}\in K$ we have
$$a_{n+d}=c_{d-1}a_{n+d-1}+\cdots+c_0a_n$$
for all $n\in\mathbb Z$.
Note that if we started with an $\mathbb N$-indexed sequence, then, after cutting off some finite initial segment, we could assume that $c_0\neq 0$, and hence we could run the linear recurrence backwards to produce a $\mathbb Z$-indexed sequence satisfying the same linear recurrence.
In characteristic zero the Skolem-Mahler-Lech theorem says that the set of indices in an $\mathbb N$-indexed linear recurrence sequence where the sequence vanishes is a finite union of arithmetic progressions.
The version for $\mathbb Z$-indexed linear recurrence sequences is that it is a finite union of cosets of subgroups of $\mathbb Z$ (see, for example, \cite{Zsequence, Zsequencecorr}, which give a proof of a more general result in this direction).
This fails in positive characteristics.
It is salvaged by:

\begin{corollary}[Skolem-Mahler-Lech]
Suppose $K$ is a field of characteristic $p$ and $a=(a_n:n\in\mathbb Z)$ is a linear recurrence sequence from $K$.
Then the vanishing set,
$$Z(a):=\{n\in\mathbb Z: a_n=0\}$$
is $p$-normal.
\end{corollary}

\begin{proof}
Using positive characteristic it is relatively straightforward to show that after writing $a$ as a finite union of linear recurrences, we may assume that we have the following closed form formula: for all $n\in \mathbb Z$,
$$a_n=\sum_{i=0}^\mu b_i\alpha_i^n$$
for some positive integer $\mu$ and some nonzero $b_1,\dots,b_\mu,\alpha_1,\dots,\alpha_\mu\in K^{\alg}$.
A detailed proof of this can be found in the Preliminaries section of~\cite{derksen}.

Let $\alpha=(\alpha_1,\dots,\alpha_\mu)$ and $L=K(\alpha)$.
Now we can pose the Skolem-Mahler-Lech problem as an isotrivial Mordell-Lang problem by setting $G=\mathbb G_m^\mu$, $\Gamma=\langle\alpha\rangle\leq G\big(L)$, and $X\subseteq G$ the closed subvariety defined by the linear equation $b_1x_1+\cdots+b_\mu x_\mu=0$.
Indeed, it follows that
$$\alpha^{Z(a)}:=\{\alpha^n:n\in Z(a)\}=X\cap\Gamma.$$
Applying Theorem~\ref{ml}, since $L$ is regular over $L\cap \mathbb F_p^{\alg}=\mathbb F_q$, if we let $F:G\to G$ be the $p$-power Frobenius endomorphism, then we get that $\alpha^{Z(a)}$ is an $F^q$-subset of~$\Gamma$.
Writing things additively (and since we may as well assume that $\Gamma$ is infinite), we have that $Z(a)$ is an $F^q$-subset of $\mathbb Z$, where $F:\mathbb Z\to\mathbb Z$ is now multiplication by $p$.
As discussed before, $F$-sets and $F^q$-sets coincide when $F$ is just multiplication by a positive integer.
So $Z(a)$ is an $F$-subset of $\mathbb Z$.
By Proposition~\ref{fset=pnorm}, $Z(a)$ is $p$-normal as desired.
\end{proof}

\bigskip
\section{When $F$ is a number, $F$-sets are $F$-automatic}
\label{section-fset-fauto}
\noindent
On the way to proving the $p$-normality of $Z(a)$, Derksen first shows that $Z(a)$ is $p$-automatic.
This leads us to ask what kind of automaticity can be deduced about $F$-sets in general, meaning when $\Gamma$ is not necessarily cyclic.
The question only makes sense when ``$F$-automaticity" makes sense, and as it stands this is only defined in the literature when $F$ is (i.e., acts as multiplication by) a number.
We therefore focus on that case for now.

Suppose $F$ is an integer bigger than $1$.
Recall that a subset $S$ of $\mathbb N$ is {\em $F$-automatic} if there is a finite automaton $\mathcal A$ which takes as inputs words on the alphabet $\{0,1,\dots,F-1\}$, such that a word is accepted by $\mathcal A$ if and only if it is a base-$F$ expansion of an element of $S$.
We point the reader to~\cite{shallit} for precision and detail.
Much of the theory has natural generalisations to subsets of $\mathbb Z$ and even $\mathbb Z^d$.
So, for example, a subset $S\subseteq \mathbb Z^d$ is {\em $F$-automatic} if there is a ($d$-dimensional) finite automaton $\mathcal A$ which takes as inputs $d$-tuples of words on the alphabet $\{+,-,0,1,\dots,F-1\}$ where each word in the tuple is of the same length, such that a $d$-tuple of words is accepted by $\mathcal A$ if and only if they form the base-$F$ expansions of an element of~$S$.
Automatic subsets of arbitrary finitely generated abelian groups were introduced in~$\S$5 of~\cite{adambell}, where the basic theory was also developed.
That is the context in which we will work.
We do not give the definition, but rather make use of the following:

\begin{fact}[Proposition~5.4 of~\cite{adambell}]
\label{preauto}
Suppose $\Gamma$ is a finitely generated abelian group and $f:\mathbb Z^d\to \Gamma$ is a surjective homomorphism.
Then $S\subseteq\Gamma$ is $F$-automatic in $\Gamma$ if and only if $f^{-1}(S)$ is $F$-automatic in $\mathbb Z^d$.
\end{fact}

\begin{theorem}
\label{fset-fauto}
Suppose $\Gamma$ is a finitely generated abelian group and $F>1$ is an integer acting by multiplication.
Then every $F$-subset of $\Gamma$ is $F$-automatic.
\end{theorem}

\begin{proof}
Fixing a surjective group homomorphism $f:\mathbb Z^d\to \Gamma$, Fact~\ref{preauto} tells us that a subset $S\subseteq\Gamma$ is $F$-automatic if and only if $f^{-1}(S)$ is.
On the other hand, it is more or less immediate from the definitions that $S\subseteq\Gamma$ is an $F$-set if and only if $f^{-1}(S)$ is.
This reduces us to the case when $\Gamma=\mathbb Z^d$.

Finite  set sums of $F$-automatic sets are $F$-automatic -- for a proof in the case of~$\mathbb N$ see~\cite[Theorem~5.6.3(d)]{shallit}, which extends easily
to $\mathbb Z^d$.
So it suffices to show that singletons, subgroups, and $F$-cycles in $\mathbb Z^d$ are $F$-automatic.

It is clear how to build an automaton that recognises any given singleton.

Subgroups are also $F$-automatic.
To see this,
 given $H\leq\mathbb Z^d$, choose generators $\lambda_1,\dots,\lambda_\ell$ for $\mathbb Z^d$ such that for some $0\leq n\leq m\leq\ell$,
\begin{itemize}
\item
$\lambda_1,\dots,\lambda_n$ generate $H$
\item
$\mathbb Z^d/H=\mathbb Z\overline{\lambda}_{n+1}\oplus\cdots\oplus\mathbb Z\overline{\lambda}_\ell$
\item
$\overline\lambda_{n+1},\dots,\overline\lambda_m$ are nontorsion, and 
\item
$\overline\lambda_{m+1},\dots,\overline\lambda_\ell$ are nonzero torsion.
\end{itemize}
It follows that
$$\displaystyle a_1\lambda_1+\cdots+a_\ell\lambda_\ell\in H \iff \bigwedge_{i=n+1}^ma_i=0 \ \ \wedge \bigwedge_{j=m+1}^\ell \ord(\lambda_j)|a_j$$
Applying Fact~\ref{preauto} with the surjection $f:\mathbb Z^\ell\to\mathbb Z^d$ induced by the choice of generators $(\lambda_1,\dots,\lambda_\ell)$, it suffices to see that the set of tuples $(a_1,\dots,a_\ell)$ satisfying the right-hand side of the above equivalence is an $F$-automatic subset of $\mathbb Z^\ell$.
It is not hard to prove that that is the case -- see~\cite[Theorem~5.4.2]{shallit} for a proof that $r\mathbb N$ is $F$-automatic in $\mathbb N$.

So it remains to show that $F$-cycles are $F$-automatic.
Consider therefore
$$C(\gamma;\delta)=\{\gamma+F^\delta\gamma+F^{2\delta}\gamma+\cdots+F^{\ell\delta}\gamma:\ell<\omega\}$$
where $\gamma\in \mathbb Z^d$ and $\delta$ is a positive integer.
As $F$ is just an integer acting by multiplication diagonally on $\mathbb Z^d$, we note that $C(\gamma;\delta)=C(1;\delta)\gamma$, where $C(1;\delta)$ is an $F$-cycle in $\mathbb Z$ and the action here of $\mathbb Z$ on $\mathbb Z^d$ is diagonal.
Now, the class of $F$-automatic subsets of $\mathbb N$ is preserved under multiplication by a natural number, see~\cite[Theorem~5.6.3(e)]{shallit}.
Another way of saying this is that the orbit of an $F$-automatic subset of $\mathbb N$ acting on $\mathbb N$ by multiplication is $F$-automatic.
The argument extends readily to show that the orbit of an $F$-automatic subset of $\mathbb Z$ under the diagonal action on $\mathbb Z^d$ is $F$-automatic.
Hence it suffices to prove that $C(1;\delta)$ is $F$-automatic in $\mathbb Z$.
In fact we have already shown in Proposition~\ref{fset=pnorm} that it is $F$-normal, and $F$-automaticity follows.
But in fact $F$-automaticity is much easier to see than that; the elements of $C(1;\delta)$ are exactly those whose base-$F$ expansions are words made up of $1$'s separated by blocks of $\delta$-many $0$'s -- and it is easy to produce the automaton that recognises such words.
\end{proof}

\bigskip
\section{$F$-spanning sets}
\label{section-spanning}
\noindent
We wish now, and for the rest of this paper, to drop the assumption that $F$ acts as multiplication by an integer.
So $F$ is an arbitrary endomorphism of the finitely generated abelian group $\Gamma$.
In order to introduce a notion of ``$F$-automaticity" in this general context we need a theory of expansions of elements of $\Gamma$ in ``base-$F$".
That is the goal of this section.

\begin{definition}[Spanning set]
\label{spanning}
We'll say that a finite subset $\Sigma $ of $\Gamma$ is an $F$-\emph{spanning set} if the following properties hold:
\begin{itemize}
\item[(i)] $\Sigma$ contains $0$ and is symmetric (i.e., if $x\in\Sigma$ then $-x\in\Sigma$);
\item[(ii)] $Fx\in \Sigma \implies x\in \Sigma $;
\item[(iii)] for all $x\in \Gamma$ there is some $m\ge 0$ and $x_{0},\ldots ,x_{m}\in \Sigma $ such that 
$$x=[x_0x_1\cdots x_m]_F:=x_{0}+Fx_{1}+\cdots +F^mx_{m};$$
\item[(iv)] for all $x_1,\dots,x_5\in \Sigma $ there exist $t,t'\in \Sigma $ such that $x_1+\cdots+x_5=t+Ft'$;
\item[(v)] If $x_1,x_2,x_3\in\Sigma$ and $x_1+x_2+x_3\in F(\Gamma)$, then there exists $t\in \Sigma$ such that $x_1+x_2+x_3=Ft$.
\end{itemize}
\end{definition}

\begin{example}
\label{examplespan}
If $k\geq4$ then $\{-k+1,-k+2,\dots,0,1,\dots,k-1\}$ is a $k$-spanning set for $\mathbb Z$.
Here  $\Gamma=\mathbb Z$ and $F$ acts as multiplication by $k$.
\end{example}

The idea is to view $\Sigma$ as the digits for $F$-expansions of elements of $\Gamma$.
So, given a word $x_0x_1\cdots x_m\in \Sigma^*$, we can view it as a base-$F$ representation of
$$[x_0x_1\cdots x_m]_F:=x_{0}+Fx_{1}+\cdots +F^mx_{m}$$
(Let us make the convention that if $\emptyset$ is the empty word on $\Sigma$ then $[\emptyset]_F:=0$.)
So condition~(iii) is saying that every element has a (possibly not unique) $F$-expansion with digits from $\Sigma$.

Condition~(iv) says that the sum of five one-digit elements is a two-digit element.
This may seem like overkill, but we can always ensure this in practice and it will make several arguments easier.
The following lemma is a version of property~(iv) for sums of elements whose $F$-expansions are longer than one-digit.

\begin{lemma}
\label{carry}
Suppose $\Sigma$ is an $F$-spanning set for $\Gamma$.
Given three words $w_1,w_2,w_3$ on $\Sigma$ of length $m$, there is a word $u$ on $\Sigma$ of length $m+1$ such that
$$[w_1]_F+[w_2]_F+[w_3]_F=[u]_F.$$
\end{lemma}

\begin{proof}
We do this by induction on $m$.
The case when $m=1$ is taken care of by condition~(iv).
Suppose $w_i=x_{i,0}\cdots x_{i,m}$ for $i=1,2,3$.
Then,
\begin{eqnarray*}
\sum_{i=1}^3[w_i]_F
&=&
\sum_{i=1}^3\big([x_{i,0}\cdots x_{i,m-1}]_F+F^m(x_{i,m})\big)\\
&=&
[u_0\cdots u_m]_F+\sum_{i=1}^3F^m(x_{i,m})\ \ \ \text{ for some $u_j\in\Sigma$ by induction}\\
&=&
[u_0\cdots u_{m-1}]_F+F^m(u_m+x_{1,m}+x_{2,m}+x_{3,m})\\
&=&
[u_0\cdots u_{m-1}]_F+F^mt+F^{m+1}t'\ \ \ \text{ for some $t,t'\in\Sigma$, by~(iv)}\\
&=&
[u_0\cdots u_{m-1}tt']_F.
\end{eqnarray*}
So $u=u_0\cdots u_{m-1}tt'$ is the word of length $m+2$ that we seek.
\end{proof}

Note that condition~(v) implies condition~(ii) if $F$ is injective (by taking $x_1=Fx$ and $x_2=x_3=0$), and while we will assume injectivity of $F$ eventually, we prefer to give these definitions in more generality.
The following lemma extends condition~(v) in exactly the same way that the previous lemma extended condition~(iv).

\begin{lemma}
\label{inF}
Suppose $\Sigma$ is an $F$-spanning set for $\Gamma$.
Given three words $w_1,w_2,w_3$ on $\Sigma$ of length $m$ such that $[w_1]_F+[w_2]_F+[w_3]_F\in F\Gamma$, there is a word $u$ on $\Sigma$ of length $m$ such that
$$[w_1]_F+[w_2]_F+[w_3]_F=F([u]_F).$$
\end{lemma}

\begin{proof}
This is again by induction on $m$ with $m=1$ being condition~(v).
Suppose $w_i=x_{i,0}\cdots x_{i,m}$ for $i=1,2,3$.
Then,
\begin{eqnarray*}
\sum_{i=1}^3[w_i]_F
&=&
\sum_{i=1}^3\big([x_{i,0}\cdots x_{i,m-1}]_F+F^m(x_{i,m})\big)\\
&=&
F([u_0\cdots u_{m-1}]_F)+\sum_{i=1}^3F^m(x_{i,m})\ \ \ \text{ by induction}\\
&=&
F\big([u_0\cdots u_{m-2}]_F+F^{m-1}(u_{m-1})\big)+F^m\big(\sum_{i=1}^3x_{i,m}\big)\\
&=&
F([u_0\cdots u_{m-2}]_F)+F^m(u_{m-1}+x_{1,m}+x_{2,m}+x_{3,m})\\
&=&
F([u_0\cdots u_{m-2}]_F)+F^mt+F^{m+1}t'\ \ \ \text{ for some $t,t'\in\Sigma$, by~(iv)}\\
&=&
F\big([u_0\cdots u_{m-2}]_F+F^{m-1}t+F^mt'\big)\\
&=&
F([u_0\cdots u_{m-2}tt']_F).
\end{eqnarray*}
So $u=u_0\cdots u_{m-2}tt'$ is the word of length $m+1$ that we seek.
\end{proof}

\begin{remark}
In both Lemma~\ref{carry} and Lemma~\ref{inF} we did not use the full strength of property~(iv), we only used that the sum of four (rather than five) one-digit elements is a two-digit element.
However the stronger statement is used in the proof of Lemma~\ref{sim}(b) below.
\end{remark}

We can always expand a given $F$-spanning set so that it contains any fixed element we want.
Indeed, this follows from:

\begin{lemma}
\label{expandspan}
Suppose $\Sigma$ is an $F$-spanning set for $\Gamma$.
Fixing $m>0$ let
$$\Sigma^{(m)}:=\{[w]_F:w\in\Sigma^*, {\rm length}(w)=m\}.$$
Then $\Sigma^{(m)}$ is an $F$-spanning set for $\Gamma$.
\end{lemma}

\begin{proof}
Property~(i) and~(iii) are clearly inherited by $\Sigma^{(m)}$, and property~(v) for $\Sigma^{(m)}$ is precisely what Lemma~\ref{inF} says.

For property~(ii) suppose $a\in\Gamma$ is such that $Fa=[w]_F$ for some $w\in\Sigma^*$ of length $m$.
Say $w=x_0\cdots x_{m-1}$.
We may assume $m>1$.
Then
$$F(a-x_1-Fx_2-\cdots-F^{m-2}x_{m-1})=x_0\in\Sigma,$$
and hence by~(ii) for $\Sigma$ we have $y:=a-x_1-Fx_2-\cdots-F^{m-2}x_{m-1}\in\Sigma$.
So $a=y+[x_1\cdots x_{m-1}]_F$.
By Lemma~\ref{carry} we get $a=[u]_F$ for some $u\in\Sigma^*$ of length~$m$.
That is, $a\in\Sigma^{(m)}$.

We prove property~(iv) by induction on $m$, the case of $m=1$ is just that $\Sigma$ satisfies property~(iv).
We are given words $w_1,\dots,w_5$ on $\Sigma$ of length $m+1$.
Write each $w_i=[w_i'x_i]$ where $w_i'$ is of length $m$ and $x_i\in\Sigma$.
Then,
\begin{eqnarray*}
\sum_{i=1}^5[w_i]_F
&=&
\sum_{i=1}^5[w_i']_F+F^m\big(\sum_{i=1}^5x_i\big)\\
&=&
[y_0\cdots y_{m-1}]_F+F([z_0\cdots z_{m-1}]_F)+F^mt+F^{m+1}(t')\\
&=&
[y_0\cdots y_{m-1}t]_F+F([z_0\cdots z_{m-1}t']_F)
\end{eqnarray*}
where the second equality is by induction and~(iv) for $\Sigma$, and the $y_j$, $z_j$ and $t,t'$ are in $\Sigma$.
\end{proof}

As it turns out we will want a bit more: that $\Sigma^{(m)}$ be an $F^m$-spanning set also.
We can get this if we assume in addition that $F$ is injective.

\begin{lemma}
\label{expandspanmore}
Suppose $F$ is injective, $\Sigma$ is an $F$-spanning set for $\Gamma$, and $m>0$.
Then $\Sigma^{(m)}$ is an $F^m$-spanning set for $\Gamma$.
\end{lemma}

\begin{proof}
Property~(i) is immediate, and property~(ii) for $(\Sigma^{(m)}, F^m)$ follows from property~(ii) for $(\Sigma^{(m)}, F)$.

For property~(iii), suppose $x\in \Gamma$.
By property~(iii) for $(\Sigma, F)$ we have $x=[x_0\cdots x_\ell]_F$ for some $x_0,\dots, x_\ell\in\Sigma$.
Write $\ell=qm+r$ for some $0\leq r<m$.
Then we have
$$x=(x_0+\cdots+F^{m-1}x_{m-1})+F^m(x_m+\cdots+F^{m-1}x_{2m-1})+\cdots+F^{qm}(x_{qm}+\cdots+F^rx_\ell).$$
So $x$ has an $F^m$-expansion with digits in $\Sigma^{(m)}$.

For property~(iv), suppose $x_1,\dots,x_5\in\Sigma^{(m)}$.
By property~(iv) for $(\Sigma^{(m)}, F)$ we have $x_1+\cdots+x_5=y+Fy'$ for some $y,y'\in\Sigma^{(m)}$.
Note that $Fy'$ has an $F$-expansion of length $m+1$ with digits in $\Sigma$, so that by Lemma~\ref{carry}, $y+Fy'$ has an $F$-expansion of length $m+2$ with digits in $\Sigma$.
Say
$y+Fy'=[t_0\cdots t_{m+1}]_F$.
So
$$x_1+\cdots+x_5=y+Fy'=(t_0+\cdots+F^{m-1}t_{m-1})+F^m(t_m+Ft_{m+1}).$$
So $x_1+\cdots+x_5$ has a $2$-digit $F^m$-expansion in terms of $\Sigma^{(m)}$.
(Note that we may assume $m>1$ so that $t_m+Ft_{m+1}\in\Sigma^{(m)}$.)

Finally, for property~(v), we will use injectivity of $F$.
Suppose $x_1,x_2,x_3\in\Sigma^{(m)}$ and $x_1+x_2+x_3=F^mx$ for some $x\in \Gamma$.
It suffices to show that $x\in\Sigma^{(m)}$.
Since $\Sigma^{(m)}$ is an $F$-spanning set and $F^mx\in F\Gamma$, we get that $x_1+x_2+x_3=Fy$ for some $y\in\Sigma^{(m)}$.
By injectivity, $F^{m-1}x=y\in\Sigma^{(m)}$.
But then $x\in\Sigma^{(m)}$ by repeated use of property~(ii) for $(\Sigma^{(m)},F)$.
\end{proof}

We now address the question of when $F$-spanning sets exist.
In particular, we will show that they exist for any finitely generated subgroup of an isotrivial commutative algebraic group -- and hence in any isotrivial Mordell-Lang context -- as long as we allow replacing $F$ by some power.
Indeed, what is required is the existence of a suitable height function.

\begin{proposition}
\label{height}
Suppose $\Gamma$ is a finitely generated abelian group with an injective endomorphism $F$.
Suppose $\Gamma$ admits a function $h:\Gamma\to[0,\infty)$ satisfying
\begin{enumerate}
\item for some $D>1$ and $\kappa>0$,
$$h(-a)\le Dh(a)+\kappa$$
and
$$h(a+b)\le D(h(a)+h(b))+\kappa$$
for any $a,b\in\Gamma$;
\item Northcott property: for every $N$ there are only finitely many $a\in A$ such that $h(a)\le N$;
\item canonicity: for some $C>1$, $h(F(a))\ge Ch(a)$ for all $a\in\Gamma$ outside of a finite set.
\end{enumerate}
Then, for some $r>0$, $\Gamma$ has an $F^r$-spanning set.
\end{proposition}

\begin{proof}
Before we proceed let us point out that injectivity implies $F(\Gamma)$ is of finite index in $\Gamma$.
Indeed, $F(\Gamma)$ and $\Gamma$ have the same rank so that $\Gamma/F(\Gamma)$ is of rank $0$, but it is also finitely generated and hence must be finite.

Now, given a positive integer $N$, let $\Sigma_N$ denote the set of elements of $\Gamma$ such that one of $x$ or $-x$ has height at most $N$. 
Then by the Northcott property $\Sigma_N$ is finite and for $N$ large we have $0\in \Sigma_N$ and $\Sigma_N$ contains all $a$ for which $h(F(a))<Ch(a)$.
Then by construction we have that $x\in \Sigma_N \implies -x\in \Sigma_N$.
Also, if $x\notin\Sigma_N$ then $h(Fx)\ge Ch(x)>N$ and $h(-Fx)\ge Ch(-x)>N$, and so $Fx\notin\Sigma_N$.
So we have verified that $(\Gamma, F, \Sigma_N)$ satisfies properties~(i) and~(ii) of Definition~\ref{spanning}.
These conditions then remain true of $(\Gamma, F^r, \Sigma_N)$ for any $r>0$.

Pick $r>0$ so that we have the inequalities:
\begin{enumerate}
\item[(i)] $(D^2N+D\kappa+\kappa)/(C^r-D)\leq N$;
\item[(ii)] $6D^5(DN+\kappa) + 5D^4 \kappa  \leq C^r N$;
\item[(iii)] $3D^2(DN+\kappa) + 2D\kappa  \leq C^r N$.
\end{enumerate}
We observe that since $C^n\to\infty$ as $n\to\infty$, such an $r$ exists.
As we have seen, by injectivity, $\Gamma/F^r\Gamma$ is a finite group.
Note that the above inequalities remain true even as we increase $N$.
So we may assume that $\Sigma_N$ contains a complete set of coset representatives for $F^r\Gamma$.

We claim that $(\Gamma, F^r,\Sigma_N)$ satisfies property~(iii) from Definition~\ref{spanning}.
Toward a contradiction, suppose that there is some $x\in \Gamma$ that does not have an $F^r$-expansion with digits from $\Sigma_N$.
Pick such an $x$ of minimal height.
Then $x\notin\Sigma_N$ since otherwise, we could take $x=x_0$ would be such an expansion.
By construction, there is some $x_0\in \Sigma_N$ such that 
$x-x_0\in F^r\Gamma$.  Then $x=x_0+F^ry$.
If $h(y)<h(x)$ then, by minimality, $y$ has an $F^r$-expansion and hence so does $x$.
Thus we may assume that $h(y)\ge h(x)$.
For the same reaon, we may assume that $y\notin\Sigma_N$.
Hence,
\begin{eqnarray*}
C^r h(x)
&\le&
C^r h(y)\\
&\le&
h(F^ry) \ \ \ \ \ \ \text{since $y\notin\Sigma_n$}\\
&\le&
h(x-x_0)\\
&\le&
Dh(x) + Dh(-x_0)+\kappa\ \ \ \ \ \ \text{by condtition~(1)}\\
&\le&
Dh(x)+D^2N+D\kappa+\kappa
\end{eqnarray*}
where in the final line we use the fact that as $x_0\in \Sigma_N$, either $h(x_0)\le N$ or $h(-x_0)\le N$, and hence in either case $h(-x_0)\le DN+\kappa$ by condition~(1).
So $(C^r-D)h(x)\le D^2N+D\kappa +\kappa$.
But $(D^2N+D\kappa+\kappa)/(C^r-D)\le N$ by~(i), and so $h(x)\le N$, contradicting the fact that $x\notin\Sigma_N$.

Next observe that a simple induction on condition~(1) gives
$$h(y_1+\cdots +y_n) \le D^{n-1}(h(y_1)+\cdots + h(y_n)) + (n-1)D^{n-2}\kappa$$
for all $y_1,\dots,y_n\in\Gamma$.
Now, if $x_1,\ldots ,x_5\in \Sigma_N$ then there is some $t\in \Sigma_N$ such that 
$x_1+\cdots +x_5-t\in F^r\Gamma$.
It follows that
$$h(x_1+x_2+\cdots + x_5 - t)  \le 6D^5(DN+\kappa) + 5D^4 \kappa \le C^r N.$$
On the other hand, if we let $t'$ be such that 
$F^rt'=x_1+\cdots +x_5-t$ then
$$h(x_1+x_2+\cdots + x_5 - t) \ge C^rh(t').$$
So $h(t')\le N$ and hence $t'\in \Sigma_N$.  
We have written $x_1+\cdots+x_5$ as the two digit element $t+F^rt'$, verifying condition~(iv) for $(\Gamma,F^r, \Sigma_N)$.

Finally, to prove condtion~(v) of an $F^r$-spanning set, suppose $x_1,x_2,x_3\in \Sigma_N$ and $x_1+x_2+x_3\in F^r(\Gamma)$.
Then 
$x_1+x_2+x_3=F^rt$ for some $t\in\Gamma$.
Note that 
$h(F^rt) \le 3D^2(DN+\kappa)+ 2D\kappa \le C^r N$ on the one hand and $h(F^rt)\ge C^rh(t)$ on the other.
So $h(t)\leq N$ and we get that $t\in \Sigma_N$, as desired.
\end{proof}

\begin{corollary}
\label{mlfsetfauto}
Suppose $G$ is a commutative algebraic group over a finite field~$\mathbb F_q$ and $F:G\to G$ is the endomorphism induced by the $q$-power Frobenius map.
Suppose $\Gamma$ is a finitely generated subgroup of $G$ that is preserved by~$F$.
Then $\Gamma$ has an $F^r$-spanning set for some $r>0$.
\end{corollary}

\begin{proof}
Write $G$ as an open subset of a closed subsets of $\mathbb{P}_n$.
There is a finite open cover $\{U_i\}_{i=1}^m$ of $G$, and homogeneous $Q_0^{(i)},Q_1^{(i)},\ldots ,Q_n^{(i)}\in\mathbb F_q[x_0,\ldots ,x_n]$ with no common zeros on $U_i$, such that for $x\in U_i$ we have
$$-x=[Q_0^{(i)}(x):Q_1^{(i)}(x):\cdots :Q_n^{(i)}(x)].$$
Similarly, there is a finite open cover $\{V_i\}_{i=1}^\ell$ of $G\times G\subset\mathbb P_n\times\mathbb P_n$ and homogeneous $P_0^{(i)},P_1^{(i)},\ldots ,P_n^{(i)}\in\mathbb F_q[x_0,\ldots ,x_n,y_0,\dots,y_n]$, with no common zeros on $V_i$, such that for all $(x,y)\in V_i$ we have
$$x+y=[P_0^{(i)}(x,y):\cdots :P_n^{(i)}(x,y)].$$
Let $D$ denote the maximum of the total degrees of all these $P$'s and $Q$'s.  

Fix a finitely generated extension $K$ of $\mathbb F_q$ such that $\Gamma\leq G(K)$.
Let $1=e_1,\ldots ,e_s$ be a basis for $K$ over $K^{\langle q\rangle}:=\{x^q\colon x\in K\}$.
We then have maps $\lambda_i:K\to K$ defined by the rule
$\displaystyle a=\sum_{i=1}^s \lambda_i(a)^q e_i$.
Let $V$ be a finite-dimensional $\mathbb{F}_q$-vector subspace of $K$ such that $e_1,\ldots ,e_s\in V$ and $K$ is the field of fractions of the $\mathbb{F}_q$-subalgebra, $A$, of $K$ generated by $V$.

For $\alpha\in K$, we define $h_0(\alpha)$ to be the smallest nonnegative integer $s$ such that $\alpha$ can be written as a ratio $a/b$ with $a,b\in V^s$.
Here $V^0:=\mathbb{F}_q$ and for $s>0$, $V^s$ denotes the $\mathbb F_q$-vector space generated by the $s$-fold products of elements in $V$.
Given $x=[x_0:\cdots :x_n]\in \mathbb{P}_n(K)$, we note that there is a representation of $x$ such that $x_0,\ldots ,x_n\in A$, and we define $h(x)$ to be $\inf\{\max\big(h_0(x_0),\ldots ,h_0(x_n)\big)\}$, where the infimum is taken as $x_0,\ldots ,x_n$ range over all representations of $x$ with each $x_i\in A$.

Notice that $h_0(x+y)\le \max(h_0(x),h_0(y))$ and $h_0(xy)\le h_0(x)+h_0(y)$,  for all $x,y\in A$.
Since the $P$'s and $Q$'s defining addition and additive inverse on $G$ have coefficients in $\mathbb F_q$ and degree bounded by $D$, it follows that
$h(x+y)\le Dh(x)+Dh(y)$ and $h(-x)\le Dh(x)$ for all $x,y\in G$. 
So $h$ satisfies property~(1) of Proposition~\ref{height} with $\kappa=0$.
That the Northcott property holds for $h$ is straightforward as each $V^i$ is a finite set.
It remains to verify canonicity.

By~\cite[Proposition 5.2]{derksen} there is some $\ell$ such that $\lambda_i(V^{d\ell })\subseteq V^{d\ell-d}$ for all $i=1,\dots,m$ and all $d>0$.
We show that $C:=\frac{\ell^2}{\ell^2-1}$ works.
Let $a\in A$ be such that $a^q\notin V^{\ell^2+\ell}$.
This omits only a finite set of elements of $A$.
Write $h_0(a^q)=\ell d+r$ for some $0\leq r<\ell$.
Then $d> \ell$ and $a^q\in V^{\ell d+r}\subseteq V^{\ell(d+1)}$.
Hence $a=\lambda_1(a^q)\in V^{(\ell-1)(d+1)}$ by choice of $\ell$.
It follows that $\frac{h_0(a^q)}{h_0(a)}\ge\frac{\ell d}{(\ell-1)(d+1)}\ge\frac{\ell^2}{\ell^2-1}$ since $d\ge \ell$.
So we have $h_0(Fa)\ge Ch_0(a)$.
It follows that $h(Fx)\ge Ch(x)$ for any $x\in \mathbb{P}_n(K)$ except for the finitely many that only have representatives with each co-ordinate having $q$th powers in $V^{\ell^2+\ell}$.

So $h$ satisfies the condition of Proposition~\ref{height}, and we can conclude the existence of an $F^r$-spanning set for some $r>0$.
\end{proof}

\bigskip
\section{Generalised $F$-automaticity}
\label{section-genfset}
\noindent
Suppose $\Gamma$ is a finitely generated abelian group and $F$ is an endomorphism of $\Gamma$.
We will introduce in this section a natural notion of $F$-automaticity extending the classical notion when $F$ is an integer, and prove that under some mild conditions $F$-sets are $F$-automatic (thus generalising Theorem~\ref{fset-fauto}).

We begin with a preliminary notion.

\begin{definition}
Suppose $\Sigma$ is an $F$-spanning set for $\Gamma$.
A subset $S\subseteq \Gamma$ is {\em $(\Sigma ,F)$-automatic} if $\{w\in\Sigma^*:[w]_F\in S\}$ is a regular language.
In other words, there is a finite automaton $\mathcal A$ which takes as inputs finite words on the alphabet $\Sigma$, such that a word $x_0x_1\cdots x_m$ is accepted by $\mathcal A$ if and only if $x_{0}+Fx_{1}+\cdots +F^mx_{m}\in S$.
\end{definition}

Our automata are reading words from left to right.
As in the classical case we want to describe this more algebraically in terms of kernels: given $S\subseteq \Gamma$, the {\em $(\Sigma,F)$-kernel} is the collection of all subsets of $\Gamma$ of the form
$$S_{x_0\cdots x_{m-1}}:=\{x\in\Gamma:x_0+Fx_1+\cdots+F^{m-1}x_{m-1}+F^mx\in S\}$$
as we range over all $m\geq 0$ and $x_0,\dots,x_{m-1}\in \Sigma$.
(Note that $S_{\emptyset}=S$.)

\begin{lemma}
\label{finker}
A subset is $(\Sigma,F)$-automatic if and only if its $(\Sigma,F)$-kernel is finite.
\end{lemma}

\begin{proof}
The classical proof extends immediately to this setting, but we give details for completeness.

Suppose $\mathcal A$ is an automaton that recognises $S$.
Let $I$ be the initial state of $\mathcal A$.
Then for any word $w=x_0\cdots x_{m-1}\in\Sigma^*$, let $\mathcal A_w$ be the automaton that is the same as $\mathcal A$ except that its initial state is $I\cdot w$, the state $\mathcal A$ ends up in upon the input $w$.
Then it is not hard to see that $\mathcal A_w$ recognises $S_w$.
Indeed;
\begin{eqnarray*}
y_0\cdots y_{r-1}\text{ is accepted by }\mathcal A_w
&\iff&
x_0\cdots x_{m-1}y_0\cdots y_{r-1}\text{ is accepted by }\mathcal A\\
&\iff&
x_0+\cdots+F^{m-1}x_{m-1}+\\
& &F^m\big(y_0+\cdots+F^{r-1}y_{r-1}\big)\in S\\
&\iff&
y_0+\cdots+F^{r-1}y_{r-1}\in S_w
\end{eqnarray*}
Since $\mathcal A$ has only finitely many states, we get only finitely many distinct automata $\mathcal A_w$ as we range over all words $w$, and hence only finitely many subsets $S_w$.

Conversely, suppose $S$ has finite $(\Sigma,F)$-kernel.
Consider the automaton whose states are the elements of the kernel, whose initial state is $S=S_\emptyset$, whose accepting states are those $S_w$ which contain $0$, and whose transitions are given by $S_w\cdot x:=S_{wx}$.
Note that for any word $w$, this automaton excepts $w$ if and only if $0\in S_w$, which happens if and only if $[w]_F\in S$.
So it witnesses that $S$ is $(\Sigma, F)$-automatic.
\end{proof}

This allows us to show that automaticity does not depend on the choice of the $F$-spanning set.

\begin{proposition}
\label{welld}
Suppose $\Sigma$ is an $F$-spanning set, $\Theta$ is an $F^r$-spanning set for some $r>0$, and $S\subseteq \Gamma$.
If $S$ is $(\Sigma,F)$-automatic then it is $(\Theta,F^r)$-automatic.

In particular, taking $r=1$ and by symmetry, $(\Sigma,F)$-automaticity does not depend on the choice of $F$-spanning set $\Sigma$.
\end{proposition}

\begin{proof}
We use the finite kernel characterisation of automaticity given by Lemma~\ref{finker}.
Let $S=S_0,\ldots ,S_t$ denote the elements of the $(\Sigma,F)$-kernel of $S$.
As $\Theta$ is finite, there is $M>0$ such that every element of $\Theta$ has an $F$-expansion with digits from $\Sigma$ of length $M$.
Let $\mathcal{T}$ denote all sets of the form $S_j-[w]_F$ where $j\in \{0,\ldots ,t\}$ and $w$ ranges over all words on $\Sigma$ of length $M$.
We claim that the $(\Theta,F^r)$-kernel of $S$ is contained in~$\mathcal{T}$.
Note that $S\in \mathcal T$ by choosing $j=0$ and $w$ the word $0^M$.
So it suffices to show that if $T\in\mathcal T$ and $z\in\Theta$ then $T_z:=\{y\in\Gamma:z+F^ry\in T\} \in\mathcal T$.
Suppose $T=S_j-[w]_F$ for some $j$ and some $w\in\Sigma^*$ of length $M$.
Write $z=[w']_F$ for some other word $w'$ on $\Sigma$ of length $M$ -- possible by the choice of $M$.
By Lemma~\ref{carry} there is $u=u_0u_1\dots u_M\in\Sigma^*$ such that $[w]_F+[w']_F=[u]_F$.
Then
\begin{eqnarray*}
T_z
&=&
\{y\in\Gamma:z+F^ry\in S_j-[w]_F\}\\
&=&
\{y\in\Gamma:[w']_F+F^ry+[w]_F\in S_j\}\\
&=&
\{y\in\Gamma:F^ry+[u]_F\in S_j\}\\
&=&
\{y\in\Gamma:[u_0\cdots u_{r-1}]_F+F^r(y+u_r+\cdots+F^{M-r}u_M)\in S_j\}\\
&=&
\{y\in\Gamma: y+u_r+\cdots+F^{M-r}u_M\in (S_j)_{u_0,\dots,u_{r-1}}\}\\
&=&
S_k-[u_r\cdots u_M]_F \ \ \ \text{ where }(S_j)_{u_0,\dots,u_{r-1}}=S_k.
\end{eqnarray*}
Note that if $M<r$ then $[u_r\cdots u_M]_F=0$.
In any case, as $r>0$ we have $[u_r\cdots u_M]_F\in\Sigma^{(M)}$, and so $T_z\in \mathcal T$.
As $\mathcal T$ is finite, we have shown that the $(\Theta,F^r)$-kernel of $S$ is finite, and hence $S$ is $(\Theta,F^r)$-automatic.
\end{proof}

\begin{definition}[$F$-automatic]
\label{defnauto}
Suppose $\Gamma$ is a finitely generated abelian group and $F$ is an endomorphism of $\Gamma$ such that $\Gamma$ admits an $F^r$-spanning set for some $r>0$.
A subset $S\subseteq \Gamma$ is defined to be {\em $F$-automatic} if for some $r>0$ and some (equivalently any by Proposition~\ref{welld}) $F^r$-spanning set $\Sigma$, $S$ is $(\Sigma,F^r)$-automatic.
\end{definition}

The following ensures that by this definition we are indeed generalising the case when $\Gamma=\mathbb Z$ and $F$ is (multiplication by) an integer.

\begin{proposition}
Suppose $\Gamma=\mathbb Z$ and $F$ acts as multiplication by an integer $k>1$.
The notion of $F$-automaticity defined here agrees with $k$-automaticity in the classical sense.
\end{proposition}

\begin{proof}
Suppose $S\subseteq \mathbb Z$ is $F$-automatic in the sense of Definition~\ref{defnauto}.
Then, for some $r>0$, and any $F^r$-spanning set $\Sigma$ of $\mathbb Z$, $S$ is $(\Sigma, F^r)$-automatic.
Note that by Lemmas~\ref{expandspanmore} and~\ref{welld}, for any $s>0$, $\Sigma^{(s)}$ is an $F^{rs}$-spanning set and $S$ is $(\Sigma^{(s)}, F^{rs})$-automatic.
So replacing $r$ with $rs$, we may assume that $k^r\geq 4$.
In that case,  we can take $\{-(k^r-1),\ldots ,0,\ldots ,k^r-1\}$ for $\Sigma$ by Example~\ref{examplespan}.
We will show that $S$ is $k^r$-automatic in the classical sense; it is well known (and easy to see) that $k^r$-automaticity implies $k$-automaticity in the classical sense.
By~\cite[Proposition~5.1]{adambell}, this is equivalent to showing that $S$ has only finitely many {\em $k^r$-kernels}.
These are just $(\Sigma, F^r)$-kernels in our sense, except that they are of the form $S_{x_0\cdots x_{m-1}}$ where all the $0\leq x_i<k^r$.
Let us call them the ``nonnegative" $(\Sigma, F^r)$-kernels.
In any case, by Lemma~\ref{finker}, there are only finitely many $(\Sigma, F^r)$-kernels; and hence there are only finitely many nonnegative ones.
So $S$ is $k^r$-automatic, and hence $k$-automatic, classically.

Suppose, for the converse, that $S$ is $k$-automatic in the classical sense.
Let $r>0$ be such that $k^r\geq 4$.
So $\Sigma:=\{-(k^r-1),\ldots ,0,\ldots ,k^r-1\}$ is an $F^r$-spanning set for $\mathbb Z$.
We show that $S$ is $(\Sigma,F^r)$-automatic.
By assumption, $S$ has only finitely many nonnegative $(\Sigma, F^r)$-kernels.
Let $\mathcal T$ be the collection of these nonnegative $(\Sigma, F^r)$-kernels together with their translations by the integer $1$.
We claim these are all the $(\Sigma, F^r)$-kernels of $S$.
We need to show that for any $T\in \mathcal T$ and any $a\in\Sigma$, the kernel $T_a\in\mathcal T$.
There are several cases to consider.

Suppose $T$ is a nonnegative $(\Sigma, F^r)$-kernel of $S$.
If  $a\ge 0$ then $T_a$ is another  nonnegative $(\Sigma, F)$-kernel of $S$, and hence is again in $\mathcal T$.
If $a<0$ then $0<a+k^r<k^r$ and
$T_a=\{x\in\mathbb Z:a+k^rx\in T\}=\{x\in\mathbb Z:(a+k^r)+k^r(x-1)\in T\}=1+T_{a+k^r}$
which by construction is also in $\mathcal T$.

Now suppose $T=1+T'$ where $T'$ is a nonnegative $(\Sigma, F^r)$-kernel of $S$.
If $a\ge1$ then $0\le a-1<k^r$ and
$$T_a=\{x\in\mathbb Z:a+k^rx\in 1+T'\}=\{x\in\mathbb Z:a-1+k^rx\in T'\}=T'_{a-1}$$
which is an element of $\mathcal T$.
If $-k^r+1<a<1$, then $-k^r< a-1<0$ and
$$T_a=(1+T')_a=T'_{a-1}$$
which is in $\mathcal T$ by the case considered in the previous paragraph.
Finally, if $a=-k^r+1$ then 
$T_a=\{x\in\mathbb Z:a+k^rx\in 1+T'\}=\{x\in\mathbb Z:-k^r+k^rx\in T'\}=1+T'_0$
which is also in $\mathcal T$.

So $S$ has only finitely many $(\Sigma, F^r)$-kernels and is thus, by Lemma~\ref{finker}, $F$-automatic in our sense.
\end{proof}

\begin{remark}
Using a similar argument, with a somewhat more complicated decomposition over quadrants, one can prove a multidimensional version of the above proposition for $\mathbb{Z}^d$ when $F$ is multiplication by $k>1$.
\end{remark}

In order to deal with the fact that elements of $\Gamma$ do not necessarily have unique $F$-expansions, the following lemma will be useful.

\begin{lemma}
\label{sim}
Suppose $F$ is injective and $\Sigma$ is an $F$-spanning set for $\Gamma$.
\begin{itemize}
\item[(a)]
Consider the equivalence relation $\sim$ on $\Sigma^*$ given by
$$w\sim u \iff {\rm length}(w)={\rm length}(u)\text{ and }[w]_F=[u]_F.$$
Then $\sim$, as a  subset of $(\Sigma\times\Sigma)^*$, is regular.
\item[(b)]
Let $G$ denote the subset of $\Sigma^*\times\Sigma^*\times\Sigma^*$ made up of triples of words $(u,v, w)$, all of the same length, satisfying $[u]_F+[v]_F=[w]_F$.
Then $G$ is regular as a subset of $(\Sigma\times\Sigma\times\Sigma)^*$. 
\end{itemize}
\end{lemma}

\begin{proof}
For part~(a) we describe a finite automaton $\mathcal A$ which takes as input the pair of words $w=x_0x_1\cdots x_{n-1}$ and $u=y_0y_1\cdots y_{n-1}$ on $\Sigma$, and reads it left to right as the word 
$(x_0,y_0) (x_1,y_1)\cdots (x_{n-1},y_{n-1})$
on $\Sigma\times \Sigma$.
We want $\mathcal A$ to accept $(w,u)$ if and only if $x_0+Fx_1+\cdots+F^{n-1}x_{n-1}=y_0+Fy_1+\cdots+F^{n-1}y_{n-1}$.

The states of $\mathcal A$ are given by all pairs $(x,y)\in\Sigma^2$.
The initial state is $(0,0)$.
The accepting states are those $(x,y)$ such that $x+Fy=0$.
If $x\notin F(\Sigma)$ then the state $(x,y)$ is a terminal rejecting state.
So assuming $x\in F(\Sigma)$, we need to say how $\mathcal A$ transitions from the state $(x,y)$ on an input $(a,b)\in\Sigma\times\Sigma$.
Choose $\hat x\in \Sigma$ such that $x=F(\hat x)$, and choose $x',y'\in\Sigma$ such that $\hat x+y+a-b=x'+F(y')$, using property~(iv) of spanning sets (Defintion~\ref{spanning}).
Then $\mathcal A$ transitions from the state $(x,y)$ to the state $(x',y')$ when it reads $(a,b)$.

Note that by property~(ii) of spanning sets, $x\in F(\Sigma)$ is equivalent to $x\in F(\Gamma)$, for all $x\in\Sigma$.

{\em Claim: If $n={\rm length}(w)={\rm length}(u)$ and on input $(w,u)$ the terminal state of $\mathcal A$ is $(x,y)$ then either $x\notin F(\Sigma)$ or $[w]_F-[u]_F=F^{n-1}x+F^ny$.}
We prove this by induction on $n$, the case of $n=0$ being vacuous.
Suppose $w=w'a$ and $u=u'b$ where $w'$ and $u'$ are of length $n$.
Let $(x',y')$ be the state of $\mathcal A$ upon input $(w',u')$ and $(x, y)$ be the state of $\mathcal A$ upon input $(w,u)$.
If $x'\notin F(\Sigma)$ then $(x',y')$ is a terminal rejecting state and so $x=x'\notin F(\Sigma)$.
So assume $x'\in F(\Sigma)$.
Then by construction we have $\hat x+y'+a-b=x+Fy$ where $x'=F(\hat x)$ and $\hat x\in \Sigma$.
So,
\begin{eqnarray*}
[w]_F-[u]_F
&=&
[w']_F-[u']_F+(F^na-F^nb)\\
&=&
F^{n-1}x'+F^n y'+(F^na-F^nb)\ \ \ \ \text{ by the inductive hypothesis}\\
&=&
F^n(\hat x+y'+a-b)\\
&=&
F^nx+F^{n+1}y
\end{eqnarray*}
as desired.

Let $(x,y)$ be the state of $\mathcal A$ upon input $(w,u)$.
We show that $(x,y)$ is an accepting state of $\mathcal A$ if and only if $[w]_F=[u]_F$.
The left-to-right direction is now clear: if $(x,y)$ is an accepting state then $x+Fy=0$ by construction, and since $x=F(-y)$, the Claim gives us that $[w]_F-[u]_F=F^{n-1}(x+Fy)=0$.

Conversely, suppose $[w]_F=[u]_F$.
If $x\in F(\Sigma)$ then by the Claim we have $0=[w]_F-[u]_F=F^{n-1}(x+Fy)$ and so by injectivity of $F$ we get $x+Fy=0$, and it follows that $(x,y)$ is an accepting state as desired.
So assume toward a contradiction that $x\notin F(\Sigma)$.
Then $(x,y)$ is a terminal rejecting state.
Let's consider the first time $\mathcal A$ gets there on input $(w,u)$.
That is, write $w=w'aw''$ and $u=u'bu''$ where the state of $\mathcal A$ on input $(w',u')$ is $(x',y')\neq (x,y)$, but on input $(w'a,u'b)$ the state is $(x,y)$.
It must be that $(x',y')$ is not terminal rejecting, so there is $\hat x\in\Sigma$ such that $x'=F(\hat x)$ and $\hat x+y'+a-b=x+Fy$.
Let $m={\rm length}(w')={\rm length}(u')$.
We have,
\begin{eqnarray*}
0
&=&
[w]_F-[u]_F\\
&=&
[w']_F-[u']_F+F^m(a-b)+F^{m+1}([w'']_F-[u'']_F)\\
&=&
F^{m-1}(x')+F^m(y')+F^m(a-b) + F^{m+1}([w'']_F-[u'']_F) \ \text{ by the Claim}\\
&=&
F^m(\hat x+y'+a-b)+ F^{m+1}([w'']_F-[u'']_F)\\
&=&
F^mx + F^{m+1}(y+[w'']_F-[u'']_F)
\end{eqnarray*}
so that $F^mx\in F^{m+1}(\Gamma)$.
By injectivity of $F$ we get that $x\in F(\Gamma)$, and hence $x\in F(\Sigma)$ by property~(ii).
This contradiction completes the proof of the Lemma.

Part~(b) has a rather similar proof.
Our automaton $\mathcal B$ that will accept $G$ is defined as follows.
The states of $\mathcal B$ are the same as those of $\mathcal A$: they are given by all pairs $(x,y)\in\Sigma^2$, with initial state $(0,0)$, and accepting states those $(x,y)$ such that $x+Fy=0$.
Again, if $x\notin F(\Sigma)$ then the state $(x,y)$ is a terminal rejecting state.
So assuming $x\in F(\Sigma)$, we need to say how $\mathcal B$ transitions from the state $(x,y)$ on an input $(a,b,c)\in\Sigma\times\Sigma\times \Sigma$.
Choose $\hat x\in \Sigma$ such that $x=F(\hat x)$, and choose $x',y'\in\Sigma$ such that $\hat x+y+a+b-c=x'+F(y')$, using property~(iv) of spanning sets (Defintion~\ref{spanning}).
Then $\mathcal B$ transitions from the state $(x,y)$ to the state $(x',y')$ when it reads $(a,b,c)$.
The corresponding claim in this case is that {\em If $(u,v,w)$ is a triple of words of length $n$, and on input $(u,v,w)$ the terminal state of $\mathcal B$ is $(x,y)$ then either $x\notin F(\Sigma)$ or $[u]_F+[v]_F-[w]_F=F^{n-1}x+F^ny$.}
The proof is by induction on $n$, is very much like part~(a), and is left to the reader.
From this claim, one shows, again exactly as in part~(a), that $\mathcal B$ works.
\end{proof}

We can now prove some closure properties for $F$-automatic sets.

\begin{proposition}
\label{preservation}
Suppose $F$ is injective.
The $F$-automatic subsets of $\Gamma$ satisfy the following closure properties:
\begin{itemize}
\item[(a)]
Finite unions of $F$-automatic sets are $F$-automatic.
\item[(b)]
If $r>0$, $\Sigma$ is an $F^r$-spanning set for $\Gamma$, and $\mathcal L\subseteq \Sigma^*$ is a regular language, then $\{[u]_{F^r}:u\in\mathcal L\}$ is $(\Sigma,F^r)$-automatic (and hence $F$-automatic).
\item[(c)]
Finite set sums of $F$-automatic sets are $F$-automatic.
\item[(d)]
If $\Gamma$ has an $F^r$-spanning set for some $r>0$ then singletons are $F$-automatic.
\item[(e)]
For any $r>0$ for which there exists an $F^r$-spanning set, $F^r$-invariant subgroups of $\Gamma$ are $F$-automatic.
\end{itemize}
\end{proposition}

\begin{proof}
Suppose $S_1$ is $(\Sigma_1, F^r)$-automatic and $S_2$ is $(\Sigma_2, F^s)$-automatic.
Let $m=rs$.
Then, by Lemma~\ref{expandspanmore}, $\Sigma_1^{(s)}$ is an $F^m$-spanning set, and by Lemma~\ref{welld}, both $S_1$ and $S_2$ are $(\Sigma_1^{(s)}, F^m)$-automatic.
Their union is $(\Sigma_1^{(s)}, F^m)$-automatic since the union of regular languages is regular.

It suffices to prove~(b) for $r=1$.
Let $S:=\{[u]_F:u\in\mathcal L\}$.
Note that the concatenation of  $\mathcal L$ with the language $0^*$, remains regular and still satisfies the above property since $[u]_F=[u0]_F$ for any word $u$.
So we may as well assume that $\mathcal L$ is closed under appending $0$'s on the right.
Now consider
$\mathcal L':=\{w\in\Sigma^*:\exists u\in\mathcal L, w\sim u\}$.
By Lemma~\ref{sim}(a) we know that $\sim$ is regular in $\Sigma^*\times\Sigma^*$, and by assumption $\mathcal L$ is regular, so from basic closure properties of regularity (see, for example, \cite[1.3]{shallit}) we get that $\mathcal L'$ is regular.
It follows that
$$\hat{\mathcal L}:=\{w\in\Sigma^*:w0^i\in\mathcal L',\text{ for some }i\geq0\}$$
is regular; this is just the quotient of the regular language $\mathcal L'$ by the regular language~$0^*$ -- see~\cite[4.1.6]{shallit}.
We now claim that $\hat{\mathcal L}=\{w\in\Sigma^*:[w]_F\in S\}$, from which it follows by definition that $S$ is $F$-automatic.
If $[w]_F\in S$ then $[w]_F=[u]_F$ for some $u\in\mathcal L$.
If ${\rm length}(u)\leq{\rm length}(w)$, then letting $i={\rm length}(w)-{\rm length}(u)$ we have that $[w]_F=[u0^i]_F$ and hence $w\sim u0^i$.
Since $u0^i\in\mathcal L$ also, we have $w\in\mathcal L'\subseteq\hat{\mathcal L}$, as desired.
If ${\rm length}(u)\geq{\rm length}(w)$ then letting $i={\rm length}(u)-{\rm length}(w)$ we have that $w0^i\sim u$, and so $w0^i\in\mathcal L'$, and so $w\in\hat{\mathcal L}$, as desired.
Conversely, suppose $w\in\hat{\mathcal L}$.
Then $w0^i\sim u$ for some $i\geq 0$ and $u\in\mathcal L$.
Hence $[w]_F=[w0^i]_F=[u]_F\in S$.

For part~(c), by the same trick as in part~(a) -- that is, using Lemmas~\ref{expandspanmore} and~\ref{welld} -- we reduce to proving that if $S_1$ and $S_2$ are both $(\Sigma,F^r)$-automatic for some $r>0$ and some $F^r$-spanning set $\Sigma$, then so is $S_1+S_2$.
For this latter statement, it suffices to consider the case when $r=1$.
We now proceed as in part~(b), except that we use Lemma~\ref{sim}(b) instead of Lemma~\ref{sim}(a).
Let $\mathcal L_i:=\{w\in\Sigma^*:[w]_F\in S_i\}$ for $i=1,2$ be the corresponding regular languages.
Let
$$\mathcal L':=\{w\in\Sigma^*:\exists u\in\mathcal L_1, \exists v\in\mathcal L_2, (u,v,w)\in G\},$$
where $G$ is as in Lemma~\ref{sim}(b).
So we have that $\mathcal L'$ is regular, and again the quotient of $\mathcal L'$ by $0^*$, let's call it $\hat{\mathcal L}$, is regular.
We claim $\hat{\mathcal L}=\{w\in\Sigma^*:[w]_F\in S_1+S_2\}$, from which it will follow that $S_1+S_2$ is $F$-automatic.
If $[w]_F\in S_1+S_2$ then $[w]_F=[u]_F+[v]_F$ for some $u\in\mathcal L_1$ and $v\in \mathcal L_2$.
If ${\rm max}({\rm length}(u),{\rm length}(v))\leq{\rm length}(w)$, then after padding $u$ and $v$ with zeroes --  which note preserves $\mathcal L_1$ and $\mathcal L_2$ -- we may assume that $u, v$, and $w$ all have the same length.
So $(u,v,w)\in G$.
Hence $w\in\mathcal L'\subseteq\hat{\mathcal L}$, as desired.
If ${\rm max}({\rm length}(u),{\rm length}(v))\geq{\rm length}(w)$ then after padding one of the $u$ and $v$ with zeroes we may assume there is an $i\geq 0$ such that $(u,v,w0^i)\in G$.
So $w0^i\in\mathcal L'$, and we get $w\in\hat{\mathcal L}$, as desired.
Conversely, suppose $w\in\hat{\mathcal L}$.
Then $(u,v,w0^i)\in G$ for some $i\geq 0$ and $u\in\mathcal L_1, v\in\mathcal L_2$.
Hence $[w]_F=[w0^i]_F=[u]_F+[v]_F\in S_1+S_2$.

Part~(d) follows from part~(b) together with the fact that singleton languages are regular.

For part~(e) we fix an $F^r$-spanning set $\Sigma$ and show that $F^r$-invariant subgroups are in fact $(\Sigma,F^r)$-automatic.
To prove this formulation, we may as well assume that $r=1$.
We first reduce to the case when the $F$-invariant subgroup $N\leq \Gamma$ has the further property that $F^{-1}(N)=N$.
Indeed, by $F$-invariance
$$N\subseteq F^{-1}(N)\subseteq F^{-2}(N)\subseteq\cdots$$
and so as $\Gamma$ is finitely generated, for some $r\geq0$, we have $F^{-r-1}(N)=F^{-r}(N)$.
Suppose we knew that $F^{-r}(N)$, which is still $F$-invariant, were $F$-automatic.
It is not hard to see that if $S$ is $F$-automatic then so is $F(S)$ -- if $\mathcal L=\{w:[w]_F\in S\}$ is a regular language then $\{u:\exists w\in\mathcal L, u=0w\}$ is also regular and witnesses $F$-automaticity of $F(S)$ by part~(b).
So $F^r(F^{-r}(N))$ is $F$-automatic.
By injectivity and finite generation, $F^r(F^{-r}(N))$ is of finite index in $F^{-r}(N)$, and hence in $N$.
That is, $N$ is a finite union of translates of $F^r(F^{-r}(N))$.
By parts~(c) and~(d), translates of $F$-automatic sets are $F$-automatic.
Hence $N$ is $F$-automatic.

So we may assume that $F^{-1}(N)=N$.
Note that $N_0:=F(\Gamma)\cap N$ has finite index in $N$, and fix coset representatives $a_1,\dots,a_r$.
Fix an $F$-spanning set $\Sigma$.
Let $m>0$ be such that each $a_i$ has an $F$-expansion with respect to $\Sigma$ of length bounded by $m$.
Let $\mathcal T$ denote the set of all finite (possibly empty) unions of cosets of $N$ of the form $N+[w]_F$ where $w\in\Sigma^*$ is of length $m$.
We claim that the $(\Sigma,F)$-kernel of $N$ is contained in the finite collection $\mathcal T$, which will prove $F$-automaticity by Lemma~\ref{finker}.
Since $N\in \mathcal T$ it suffices to show that if $T\in \mathcal T$ and $x\in\Sigma$ then $T_x=\{a\in\Gamma:x+Fa\in T\}\in\mathcal T$.
Note that $T=\bigcup_{j=1}^\ell T_j$ implies $T_x=\bigcup_{j=1}^\ell (T_j)_x$, so we may assume that $T=N+[w]_F$ for some word $w$ of length $m$.
Then
\begin{eqnarray*}
T_x
&=&
\{a:x+Fa\in N+[w]_F\}\\
&=&
\bigcup_{i=1}^r \{a : x+Fa \in N_0 +a_i +[w]_F\}\\
&=&
\bigcup_{i=1}^r F^{-1}(N_0 +a_i +[w]_F-x).
%&=&
%\bigcup_{i=1}^r F^{-1}(N_0 +[u_i]_F)
\end{eqnarray*}
Now, since $N_0\subseteq F(\Gamma)$, if $a_i +[w]_F-x\not \in F(\Gamma)$ then $F^{-1}(N_0 +a_i +[w]_F-x)$ is empty.
So we need only consider those $i$ such that $a_i +[w]_F-x\in F(\Gamma)$.
By Lemma~\ref{inF}, in this case there is $v_i$ of length $m$ such that $a_i +[w]_F-x=F([v_i]_F)$.
But note that $F^{-1}(N_0 +F([v_i]_F))=N+[v_i]_F$.
Indeed, if $a= n+[v_i]_F$ with $n\in N$, then $F(a)=F(n)+F([v_i]_F)$ and $F(n)\in F(\Gamma)\cap N=N_0$ by $F$-invariance.
Conversely, if $F(a)=n_0+F([v_i]_F)$ with $n_0\in F(\Gamma)\cap N$, then $n_0=F(g)$ for some $g\in F^{-1}(N)=N$, and so by injectivity of $F$ we have $a=g+[v_i]_F$.
We have proved that $T_x$ is a finite union of sets of the form $N+[v_i]_F$ with $v_i$ of length $m$ -- i.e., $T_x\in\mathcal T$ as desired.
\end{proof}

The following generalises Theorem~\ref{fset-fauto} well beyond the case when $F$ is multiplication by an integer.

\begin{theorem}
\label{generalfsetfauto}
Suppose $\Gamma$ is a finitely generated abelian group equipped with an injective 
endomorphism $F$, such that for any $\delta>0$, $F^{\delta}-1$ is not a zero divisor in the subring of $\End(\Gamma)$ generated by $F$.
Suppose, moreover, that an $F^r$-spanning set for $\Gamma$ exists for some $r>0$.
Then every $F$-subset of $\Gamma$ is $F$-automatic.
\end{theorem}

\begin{proof}
By parts~(a), (c), (d), (e) of Proposition~\ref{preservation}, it only remains to show that every $F$-cycle is $F$-automatic.
Fix $r>0$ so that an $F^r$-spanning set exists.
By Fact~\ref{rmult}, which is where the assumption on $F^\delta-1$ is used, every $F$-cycle is a finite union of translates of $F^r$-cycles.
So it suffices to show that for some $F^r$-spanning set, $\Sigma$, all $F^r$-cycles are $(\Sigma,F^r)$-automatic.
To prove this latter statement, it suffices to consider the case when $r=1$.
Fix an $F$-spanning set $\Sigma$ and consider an $F$-cycle
$$S=C(\gamma,\delta)=\{\gamma+F^\delta\gamma+\cdots+F^{\ell\delta}\gamma:\ell<\omega\}$$
where $\gamma\in \Gamma$ and $\delta>0$.
Write $\gamma=[w]_F$ for some $w\in\Sigma^*$ of length $m>0$.
By Lemma~\ref{expandspan}, $\Sigma^{(m)}$ is another $F$-spanning set that now contains $\gamma$.
Replacing $\Sigma$ by $\Sigma^{(m)}$, and setting $\mathcal L$ to be the regular language $(\gamma\underbrace{0\cdots0}_{\delta-1\text{ times}})^*\gamma$, we see that $S=\{[w]_F:w\in\mathcal L\}$.
Hence $S$ is $(\Sigma,F)$-automatic by Proposition~\ref{preservation}(b).
\end{proof}

See Proposition~\ref{height} for a sufficient condition for when an $F^r$-spanning set exists for some $r>0$.
In particular, this theorem does apply to the diophantine-geometric context we are interested in.

\begin{corollary}
\label{corfsetfauto}
Suppose $G$ is a connected commutative algebraic group over a finite field~$\mathbb F_q$ and $F:G\to G$ is the endomorphism induced by the $q$-power Frobenius.
Suppose $\Gamma$ is a finitely generated subgroup of $G$ that is preserved by~$F$.
Then every $F$-subset of $\Gamma$ is $F$-automatic.

In particular, if $G$ is semiabelian and  $X\subseteq G$ is a  closed subvariety then $X\cap \Gamma$ is $F$-automatic.
\end{corollary}

\begin{proof}
We first verify that $F^{\delta}-1$ is not a zero divisor in $\End(G)$ for any $\delta>0$.
Indeed, suppose $F^\delta \phi=\phi$ where $\phi$ is an algebraic endomorphism of $G$.
Then, letting $L$ be a finitely generated extension of $\mathbb F_q$ over which $\phi$ is defined and such that $G(L)$ is Zariski dense in $G$, we see that $\phi\big(G(L)\big)\subseteq\bigcap_{n<\omega}G(L^{q^{\delta n}})$, and the latter set is finite.
So in fact $\phi$ maps all of $G$ to a finite set, and $\phi=0$ by connectedness.

By Corollary~\ref{mlfsetfauto}, $\Gamma$ has an $F^r$-spanning set for some $r>0$.

So $(\Gamma,F)$ satisfies the conditions of Theorem~\ref{generalfsetfauto}, and all $F$-sets are $F$-automatic.

The ``in particular" clause follows now from the Isotrivial Mordell-Lang theorem of~\cite{fsets}, see Theorem~\ref{ml} above.
\end{proof}

\bigskip
\section{Generalised $F$-normality}
\label{section-fnormal}
\noindent
In this final section we introduce a natural notion of $F$-normality capturing an aspect of Derksen's notion when $F$ is an integer.
We will show that $F$-sets are $F$-normal (thus generalising one direction of Proposition~\ref{fset=pnorm}).

We first recall the notion of a sparse language.
Let $\Sigma$ be an alphabet and let $\mathcal{L}\subseteq \Sigma^*$ be a language.  Define 
$$f_{\mathcal{L}}(n):=\#\{w\in \mathcal{L}\colon {\rm length}(w)\le n\}.$$
The language $\mathcal L$ is called {\em sparse} if it is regular and any of the equivalent conditions of Proposition~\ref{sparse} hold.  Sparse regular languages have been studied extensively; we give a collection of equivalent formulations of sparseness, by combining results from
\cite{Ginsburg&Spanier:1966,Trofimov:1981,Ibarra&Ravikumar:1986,Szilard&Yu&Zhang&Shallit:1992,Gawrychowski&Krieger&Rampersad&Shallit:2010}.

\begin{proposition}
\label{sparse}
Let $\mathcal{L}$ be a regular language. The following are equivalent:
\begin{enumerate}
\item  $f_\mathcal{L}(n)={\rm O}(n^d)$ for some natural number $d$.   
\item $f_{\mathcal{L}}(n)={\rm o}(C^n)$ for every $C>1$.
\item There do not exist words $u,v,a,b$ with $a,b$ non-trivial and of the same length and $a\neq b$ such that $u\{a,b\}^* v\subseteq \mathcal{L}$.
\item Suppose $\Gamma = (Q,\Sigma,\delta,q_0,F)$ is an automaton accepting $\mathcal{L}$ in which all states are accessible.
Then $\Gamma$ satisfies the following.
\begin{itemize}
\item[($*$)] If $q$ is a state such that $\delta(q,v)\in F$ for some word $v$ then there is at most one non-trivial word $w$ with the property that $\delta(q,w)=q$ and $\delta(q,w')\neq q$ for every non-trivial proper prefix $w'$ of $w$;
\end{itemize}
\item There exists  an automaton accepting $\mathcal{L}$ that satisfies~{\em($*$)}.
\item $\mathcal{L}$ is a finite union of languages of the form $v_1 w_1^* v_2 w_2^* \cdots v_k w_k^* v_{k+1}$ where $k\ge 0$ and the $v_i$ are possibly trivial words and the $w_i$ are non-trivial words.
\end{enumerate}
\end{proposition}

\begin{proof}  We observe that (1) $\implies$ (2) follows immediately from the definition.

If~(3) does not hold then we have non-trivial words $a$ and $b$ of the same length such that $u\{a,b\}^* v\subseteq \mathcal{L}$.  Let $D=\max(|a|,|b|,|u|,|v|)$.
Then we see that we have at least $2^m$ distinct words of length $\leq(m+2)D$.
Thus $f_{\mathcal{L}}((m+2)D) \ge 2^m$.
Hence $f(n)\ge \frac{1}{4}(2^{\frac{1}{D}})^n$ for infinitely many $n$.
So (2) does not hold; thus (2) implies (3).

If (4) does not hold, then there is some minimal automaton $\Gamma= (Q,\Sigma,\delta,q_0,F)$ accepting $\mathcal{L}$ such that for some state $q$ such that $\delta(q,v)$ is an accepting state and we have two words distinct words $w,w'$ such that $\delta(q,w)=\delta(q,w')=q$ and no proper non-trivial prefix of $w$ or $w'$ sends $q$ to itself.  Since our automaton is minimal there is some word $u$ such that $\delta(q_0,u)=q$, where $q_0$ is our initial state. Then we see that $u\{w,w'\}^* v \subseteq \mathcal{L}$.  We let $a=w^{{\rm length}(w')}$ and $b=(w')^{{\rm length}(w)}$.  Then $a,b$ are distinct words of the same length and we have $u\{a,b\}^* v\subseteq \mathcal{L}$.
So (3) implies (4).

That (4) implies (5) is clear.

We can show that (5) implies (6) by induction on the number of states that admit a path to an accepting state in an automaton accepting $\mathcal L$ and satisfying~($*$).
Note that if $\mathcal L$ is accepted by an automaton with only one state then it cannot satisfy~($*$), unless $|\Sigma|=1$ or $\mathcal L=\emptyset$ in both of which cases the proposition is vacuously true.
Now suppose that (5) holds and $\Gamma = (Q,\Sigma,\delta,q_0,F)$ is an automaton accepting $\mathcal{L}$ and satisfying~($*$).
Let us introduce some notation.
Given two states $q,q'$, we'll say that $q\prec q'$ if there is some $u\in \Sigma^*$ such that $\delta(q,u)=q'$.
If $q\prec q'$ and $q'\prec q$, we'll write $q\sim q'$.
We note that $\sim$ is an equivalence relation and $\prec$ induces a partial order on the equivalence classes and we'll let $[q]$ denote the equivalence class of $q$.

Now fix a state $q$ such that $q\prec q'$ for some $q'\in F$.
Let $\{q=q_1,\ldots ,q_d\}=[q]$.
Note that every state visited on a path from $q$ to any $q_i$ is in $[q]$.
Also, by~($*$), there is a unique path that goes from $q$ to $q_d$ and back (without going through $q$ along the way), and that path must pass through all the other $q_i$.
So, after possibly relabelling, there are unique $s_1,\ldots ,s_d\in \Sigma$ such that $\delta(q_i,s_i)=q_{i+1}$, where we take $q_{d+1}=q_1$.
Now our language $\mathcal{L}$ can be written as $\mathcal{L}'\cup \mathcal{L}''$ where $\mathcal{L}'$ is the collection of words accepted by our automaton without ever visiting a state in $[q]$ and $\mathcal{L}''$ is the collection of accepted words upon which input the automaton passes through at least one state in $[q]$.

The language $\mathcal{L}'$ is accepted by the automaton obtained from $\Gamma$ by excising $[q]$ from $\Gamma$ and replacing it with a terminal rejecting state.
Note that this automaton has fewer states that lead to an accepting state than $\Gamma$ did.
As the new smaller automaton still satisfies~($*$), our induction hypothesis implies that $\mathcal{L}'$ is of the desired form.

Now consider $\mathcal{L}''$.
For each $i$ let $\{(r_{ik},a_{ik}):k=1,\dots,n_i\}$ be the set of all pairs where $r_{ik}$ is a state not in $[q]$ and $a_{ik}\in \Sigma$ is such that $\delta(r_{ik},a_{ik})=q_i$.
Let $A_{ik}$ be the set of all words that lead from $q_0$ to $r_{ik}$ without visiting any state in $[q]$.
Similarly, define $\{(s_{ik},b_{ik}):k=1,\dots,m_i\}$ be the set of all pairs where $s_{ik}$ is a state not in $[q]$ and $b_{ik}\in \Sigma$ is such that $\delta(q_i,b_{ik})=s_{ik}$.
Let $B_{ik}$ be the set of all words that lead from $s_{ik}$ to an accepting state without visiting any state in $[q]$.
Finally, for each $(i,j)$ let $u_{ij}$ be the unique minimal (possibly trivial) path from $q_i$ to $q_j$.
We have,
$$\mathcal{L}''=\bigcup_{i,j=1}^d \bigcup_{k=1}^{n_i}\bigcup_{\ell=1}^{m_j}A_{ik}a_{ik}(s_i\cdots s_d s_1\cdots s_{i-1})^* u_{ij} b_{j\ell}B_{j\ell}.$$ 
It therefore suffices to show that each $A_{ik}$ and $B_{j\ell}$ are of the desired form.
But $A_{ik}$ is accepted by the automaton obtained from $\Gamma$ by excising $[q]$, replacing it with a terminal rejecting state $z$, and making $r_{ik}$ the only accepting state and from which all  letters lead to $z$.
This new automaton still satisfies~($*$) and has fewer states that lead to an accepting state than $\Gamma$, since in $\Gamma$ if $r\prec r_{ik}$ then $r\prec q'\in F$.
So $A_{ik}$ is of the desired form by the induction hypothesis.
For $B_{j\ell}$, we note that it is accepted by the automaton obtained from $\Gamma$ by excising $[q]$, replacing it with a terminal rejecting state $z$, and making $s_{j\ell}$ the initial state.
Again this automaton has strictly fewer states that lead to accepting states, and it satisfies~($*$), so by the induction hypothesis $B_{j\ell}$ is also of the desired form.
Thus we obtain (6).

Finally, to see that (6) implies (1), first observe that a finite union of languages satisfying~(1) will again satisfy~(1).
Next, note that~(1) is preserved by concatenation as well since $f_{\mathcal{L}\mathcal{L'}}(n) \le  f_{\mathcal{L}}(n) f_{\mathcal{L'}}(n)$.
We thus reduce to proving that a singleton $\{v\}$ satisfies~(1) and a language of the form $w^*$ satisfies~(1).
Both of these claims are immediate.
\end{proof}

\begin{definition}[$F$-normal]
Suppose $\Gamma$ is a finitely generated abelian group and $F$ is an endomorphism of $\Gamma$.
A subset $S\subseteq \Gamma$ is {\em $F$-sparse} if $S=\{[w]_{F^r}:w\in\mathcal L\}$ for some $r>0$, some $F^r$-spanning set $\Sigma$, and some sparse language $\mathcal L\subset\Sigma^*$.
An {\em $F$-normal} subset is one that up to a finite symmetric difference is a finite union of sets of the form $\gamma+T+H$ where $\gamma\in \Gamma$, $T$ is $F$-sparse, and $H$ is an $F$-invariant subgroup.
\end{definition}

Note that by Proposition~\ref{preservation}, if $F$ is injective then $F$-normal sets are $F$-automatic.

\begin{remark}
It seems likely that $F$-sparseness is preserved under translation, in which case the definition of $F$-normality could be simplified.
However, this property is not entirely obvious, and not necessary for our purposes here.
We delay verifying it until some future work in which it may be desired.
\end{remark}

\begin{theorem}
\label{fset=fnorm}
Suppose $\Gamma$ is a finitely generated abelian group equipped with an injective 
endomorphism $F$, such that $F^{\delta}-1$ is not a zero divisor in the subring of $\End(\Gamma)$ generated by $F$ for any $\delta>0$, and $\Gamma$ has an $F^r$-spanning set for some $r>0$.
Then the $F$-subsets of $\Gamma$ are $F$-normal.\footnote{We are grateful to Christopher Hawthorne who found a mistake in the proof of an earlier version claiming also that every $F$-normal set is, up to a finite symmetric difference, an $F$-set.}
\end{theorem}

\begin{proof}
Every $F$-set is a finite union of sets of the form
$\gamma_0+C(\gamma_1;\delta_1)+\cdots+C(\gamma_d;\delta_d)+H$
where $\gamma_0,\dots,\gamma_d\in \Gamma$, $H$ is an $F$-invariant subgroup of $\Gamma$, and $\delta_1,\dots,\delta_d$ are positive integers.
Using Fact~\ref{rmult} and Lemma~\ref{expandspanmore}, we may take $\delta_1=\dots=\delta_d=:\delta$ such that an $F^\delta$-spanning set exists.
Comparing definitions, it suffices therefore to show that
$$T=C(\gamma_1;\delta)+\cdots+C(\gamma_d;\delta)$$
is $F$-sparse.
For each permutation $\sigma\in S_d$, let
$I_\sigma:=\{(n_1,\dots,n_d)\in\mathbb N^d: n_{\sigma(1)}\leq n_{\sigma(2)}\leq \cdots \leq n_{\sigma(d)}\}$
and let $T^\sigma$ be the set
$$\{(\gamma_1+F^\delta \gamma_1+\cdots + F^{n_1\delta}\gamma_1) +\cdots + (\gamma_\ell+F^\delta \gamma_\ell+\cdots + F^{n_d\delta}\gamma_d) \ : (n_1,\ldots ,n_d)\in I_\sigma\}.$$
Then $\displaystyle T=\bigcup_{\sigma\in S_d}T^\sigma$, and so as $F$-sparse sets are closed under finite unions it suffices to show that each $T^\sigma$ is $F$-sparse.
By symmetry it suffices to consider the case when $\sigma=\id$.
Let $u_i:=\gamma_1+\cdots + \gamma_i$ for $1\le i\le d$, and use Lemma~\ref{expandspan} to get an $F^\delta$-spanning set that contains $u_1,\dots,u_d$.
Now, observe that $T^{\id} = \{[w]_{F^\delta} \colon w\in \mathcal{L}\}$ 
where $\mathcal{L}=u_d u_d^*u_{d-1}^*\cdots u_1^*$.
This is a sparse language by Proposition~\ref{sparse}(6).
\end{proof}

\begin{corollary}
\label{corfset=fnorm}
Suppose $G$ is a connected commutative algebraic group over a finite field~$\mathbb F_q$ and $F:G\to G$ is the endomorphism induced by the $q$-power Frobenius.
Suppose $\Gamma$ is a finitely generated subgroup of $G$ that is preserved by~$F$.
Then the $F$-subsets of $\Gamma$ are $F$-normal.
In particular, if $G$ is semiabelian and  $X\subseteq G$ is a  closed subvariety then $X\cap \Gamma$ is $F$-normal.
\end{corollary}

\begin{proof}
We have already verified in the proof of Corollary~\ref{corfsetfauto} that $F^{\delta}-1$ is not a zero divisor in $\End(G)$ for any $\delta>0$.
By Corollary~\ref{mlfsetfauto}, $\Gamma$ has an $F^r$-spanning set for some $r>0$.
So $(\Gamma,F)$ satisfies the conditions of Theorem~\ref{fset=fnorm}.
The ``in particular" clause follows again from the Isotrivial Mordell-Lang theorem of~\cite{fsets}, which is Theorem~\ref{ml} above.
\end{proof}

\bigskip
%\bibliographystyle{plain}
%\bibliography{fsets-SML}

\end{document}